\documentclass[a4paper]{amsart}
\usepackage{mathrsfs}
\usepackage{amsfonts}
\usepackage{amssymb}
\usepackage{amsxtra}
\usepackage{dsfont}

\usepackage[english,polish]{babel}

\newcommand{\pl}[1]{\foreignlanguage{polish}{#1}}

\usepackage{color}
\usepackage{enumerate}

\usepackage[colorlinks,citecolor=blue,urlcolor=blue]{hyperref}

\theoremstyle{plain}
\newtheorem{theorem}{Theorem}[section]
\newtheorem{proposition}{Proposition}[section]

\newtheorem{lemma}{Lemma}[section]
\newtheorem{remark}{Remark}[section]
\theoremstyle{definition}

\numberwithin{equation}{section}

\theoremstyle{plain}
\newcounter{thm}

\newtheorem{main_theorem}[thm]{Theorem}

\newcommand{\RR}{\mathbb{R}}
\newcommand{\ZZ}{\mathbb{Z}}
\newcommand{\TT}{\mathbb{T}}
\newcommand{\CC}{\mathbb{C}}
\newcommand{\NN}{\mathbb{N}}
\newcommand{\QQ}{\mathbb{Q}}

\newcommand{\boldB}{\mathbf{B}}
\newcommand{\boldA}{\mathbf{A}}

\newcommand{\calC}{\mathcal{C}}
\newcommand{\calP}{\mathcal{P}}
\newcommand{\calQ}{\mathcal{Q}}
\newcommand{\calF}{\mathcal{F}}

\newcommand{\calO}{\mathcal{O}}

\newcommand{\calT}{\mathcal{T}}
\newcommand{\calN}{\mathcal{N}}

\newcommand{\seq}[2]{{#1}: {#2}}
\newcommand{\ind}[1]{{\mathds{1}_{{#1}}}}
\newcommand{\dist}{\operatorname{dist}}
\newcommand{\lcm}{\operatorname{lcm}}

\newcommand{\supp}{\operatorname{supp}}

\renewcommand{\atop}[2]{\substack{{#1}\\{#2}}}
\newcommand{\norm}[1]{{\left\lvert #1 \right\rvert}}
\newcommand{\sprod}[2] {{#1 \cdot #2}}

\newcommand{\abs}[1]{{\lvert {#1} \rvert}}
\newcommand{\sabs}[1]{{\left\lvert {#1} \right\rvert}}
\newcommand{\vnorm}[1]{{\left\lVert {#1} \right\rVert}}

\newcommand{\vrho}{\varrho}

\title[Discrete operators of Radon type]
{$\ell^p(\mathbb Z^d)$-estimates for discrete operators of Radon type:\\
 Maximal functions and vector-valued estimates}

\author{Mariusz Mirek}

\address{Mariusz Mirek \\
Department of Mathematics,
Rutgers University,
Piscataway, NJ 08854, USA \&
	Instytut Matematyczny\\
	Uniwersytet \pl{Wroc{\lll}awski}\\
	Plac Grun\-waldzki 2/4\\
	50-384 \pl{Wroc{\lll}aw}\\
	Poland}
\email{mariusz.mirek@rutgers.edu}

\author{Elias M. Stein}
\address{
	Elias M. Stein\\
	Department of Mathematics\\
	Princeton University\\
	Princeton\\
	NJ 08544-100, USA}
\email{stein@math.princeton.edu}

\author{Bartosz Trojan}
\address{
    Bartosz Trojan\\
    Instytut Matematyczny Polskiej Akademii Nauk\\
    ul. \pl{{\'S}niadeckich} 8\\
    00-656 Warszawa\\
    Poland}
\email{btrojan@impan.pl}

\thanks{ Mariusz Mirek was supported by Foundation for Polish Science
  - START scholarship and by NCN grant DEC-2015/19/B/ST1/01149.  Elias
  M. Stein was partially supported by NSF grant DMS-1265524.  }

\begin{document}
\selectlanguage{english}

\begin{abstract}
We prove $\ell^p\big(\mathbb Z^d\big)$ bounds, for $p\in(1, \infty)$,
of discrete maximal functions corresponding to averaging operators and
truncated singular integrals of Radon type, and their applications to
pointwise ergodic theory. Our new approach is based on a unified
analysis of both types of operators, and also yields an extension to
the vector-valued form of these results.
\end{abstract}

\maketitle

\section{Introduction}
The aim of this paper\footnote{A more detailed version of this paper
is available on arXiv: \texttt{https://arxiv.org/abs/1512.07518}} is
to study discrete analogues of operators of Radon type. A wide class
of interesting questions at the interface of harmonic analysis, number
theory and ergodic theory arise when discrete averaging operators or
singular integrals modeled on polynomial mappings are studied. The
approach undertaken in this paper has the merit of unifying results of
two major streams of the theory, the Bourgain maximal theorems, and
the theorem of Ionescu--Wainger for singular Radon transforms. In
particular, we obtain analogous results for the maximal truncated
singular Radon transforms, and for both not only in the scalar-valued
case but also in a vector-valued version.

To begin, assume that $K \in \calC^1\big(\RR^k \setminus \{0\}\big)$ is a
Calder\'{o}n--Zygmund kernel satisfying the differential inequality
\begin{align}
	\label{eq:3}
	\norm{y}^k \abs{K(y)} + \norm{y}^{k+1} \norm{\nabla K(y)} \leq 1
\end{align}
for all $y \in \RR^k$ with $\norm{y} \geq 1$ and the cancellation
condition
\begin{align}
	\label{eq:4}
	\sup_{\lambda \geq 1}
	\Big\lvert
	\int\limits_{1 \leq \norm{y} \leq \lambda} K(y) {\: \rm d} y
	\Big\rvert
	\leq 1.
\end{align}
Let 
\[
	\calP=\big(\calP_1,\ldots, \calP_{d_0}\big): \ZZ^{k} \rightarrow \ZZ^{d_0}
\]
be a polynomial mapping, where for each $j \in \{1, \ldots, d_0\}$ the
function $\calP_j:\ZZ^{k} \rightarrow \ZZ$ is a polynomial of $k$ variables with
integer coefficients  such that $\calP_j(0) = 0$. Among other things, we are interested in discrete truncated singular Radon transforms
\begin{align}
  \label{eq:24}
  T_N^{\calP} f(x)
=
\sum_{y\in\ZZ_N^k\setminus\{0\}} f\big(x - \calP(y)\big) K(y)
\end{align}
defined for a finitely supported function $f: \ZZ^{d_0} \rightarrow
\CC$, where $\ZZ_N^k=\{-N,\ldots,-1,0,1,\ldots,N\}^k$. 

Our starting point was the proof of the following theorem.
\begin{main_theorem}
	\label{thm:0}
        For every $p \in (1, \infty)$ there is $C_p > 0$ such that for all
		$f \in \ell^p\big(\ZZ^{d_0}\big)$ we have
		$$
		\big\lVert
		\sup_{N\in\NN} \big\lvert T_N^\calP f \big\rvert
		\big\rVert_{\ell^p}
		\leq
		C_p \vnorm{f}_{\ell^p}.
		$$
		Moreover, the constant $C_p$ is independent of the  coefficients of the polynomial mapping 
		$\calP$.
\end{main_theorem}
Once this theorem was proven, it became clear that the approach  used
gives a different proof of the Bourgain theorem, and ultimately  the
vector-valued versions given in Theorems \ref{thm:10} and \ref{thm:11} below.

Theorem \ref{thm:0} generalizes recent result of Ionescu and Wainger
\cite{iw}, where  $\ell^p\big(\ZZ^{d_0}\big)$
boundedness was shown, for any $p\in(1, \infty)$, for the discrete singular
Radon transform
\begin{align*}
  T^{\calP}f(x)=\sum_{y\in\ZZ^k\setminus\{0\}} f\big(x - \calP(y)\big) K(y).
\end{align*}

Theorem \ref{thm:0} also has applications in pointwise ergodic
theory, which to a large extent motivates the present work.  Namely, let
$(X, \mathcal{B}, \mu)$ be a $\sigma$-finite measure space with a
family of invertible commuting and measure preserving transformations
$S_1, S_2,\ldots,S_{d_0}$. An ergodic counterpart of $T^\calP_N$
can be defined as follows
\begin{align}
  \label{eq:17}
	\calT^\calP_N f(x)
	= \sum_{y \in \mathbb Z_N^k \setminus \{0\}} 
	f\big(S_1^{\calP_1(y)}\circ S_2^{\calP_2(y)} \circ \ldots \circ S_{d_0}^{\calP_{d_0}(y)} x\big)
	K(y).
\end{align}
We emphasize that the pointwise convergence of $T^\calP_Nf$ defined in \eqref{eq:24} can be easily obtained for any
function $f \in \ell^p\big(\ZZ^{d_0}\big)$ with $p\in(1, \infty)$ as a simple application of H\"older's inequality
and condition \eqref{eq:3}. However, the things are more complicated for the operators \eqref{eq:17} defined
on an abstract measure space $X$. We establish the following theorem. 
\begin{main_theorem}
	\label{thm:1}
	Assume that $p \in (1, \infty)$ and $K \in \calC^1\big(\RR^k \setminus \{0\}\big)$ is a
        Calder\'{o}n--Zygmund kernel satisfying \eqref{eq:3} and
        \begin{align}
          \label{eq:22}
        \int_{Q_t \setminus Q_{t'}} K(y){\: \rm d}y=0, \quad \text{for all}\quad 0< t' \le t,    
        \end{align}
where $Q_t=[-t, t]^k$. Then for every $f \in L^p(X, \mu)$ there exists
	$f^*\in L^p(X, \mu)$ such that
	\begin{align}
          \label{eq:23}
		\lim_{N\to\infty} \calT_{N}^{\calP}f(x)=f^*(x)
	\end{align}
	$\mu$-almost everywhere on $X$.
\end{main_theorem}

Theorem \ref{thm:1} can be thought of as an extension of Cotlar's
ergodic theorem \cite{cot} proving pointwise convergence for the
ergodic Hilbert transforms, which correspond to the operators
\eqref{eq:17} for $k=d_0=1$ with $\calP(x)=x$. Now the limit \eqref{eq:23}
allows us to define an ergodic singular Radon transform by setting
\[
	\calT^\calP f(x)
	= 
	\lim_{N \to \infty} \calT^\calP_N f(x).
\]
In view of the Calder\'on transference principle, the proof of Theorem
\ref{thm:1} follows from Theorem \ref{thm:0} and an oscillation inequality
on $L^2(X, \mu)$ for $\mathcal T_N^{\calP}$. Using condition \eqref{eq:22} and  arguments taken  from the
proof of Theorem \ref{thm:0}, the oscillation inequality can be easily
obtained by the methods from \cite{bou}. Alternatively, we can
mimic the proof in \cite{mt3} of $r$-variational estimates on $L^2$
leading to \eqref{eq:23} for $p=2$, which combined with
the maximal inequality yields \eqref{eq:23} for all $p\in(1, \infty)$.  The
$r$-variation estimates for the discrete operators of Radon type with
sharp ranges of parameters are the subject of \cite{mst2}.

The Bourgain-type theorems deal with the  Radon averaging operators
\begin{align}
\label{eq:28}
	 M_N^{\calP} f(x)
	=N^{-k} \sum_{y\in\NN_N^k}
	f\big(x-\calP(y)\big)
\end{align}
defined for any finitely supported function
$f: \ZZ^{d_0} \rightarrow \CC$, where
$\NN^k_N = \{1, 2, \ldots, N\}^k$. Being motivated by some questions
in the pointwise ergodic theory, Bourgain introduced and further
studied the maximal Radon transform associated to $M_N^{\calP} f$ in
\cite{bou1, bou2} and \cite{bou}. Specifically, in \cite{bou} he
proved that $\sup_{N\in\NN}|M_N^{\calP}f|$ is bounded on $\ell^p(\ZZ)$
for any $1<p\le \infty$, which corresponds to $k = d_0 = 1$. For
higher dimensional cases with general $k\ge1$ and $d_0\ge 1$ we refer
to \cite{mt3}. Quasi-invariant analogues of \eqref{eq:28}, with
polynomials of degree at most two, were also considered in
\cite{imsw}.

In this article we extend Bourgain's theorem to the vector-valued case.
Let 
\[
 \ell^p\big(\ell^2\big(\ZZ^{d_0}\big)\big)
=\Big\{(f_t: t\in\NN): \Big\|\big(\sum_{t\in\NN}|f_t|^2\big)^{1/2}\Big\|_{\ell^p}<\infty\Big\}.
\]
The main results of this paper can then be stated as Theorem \ref{thm:10} and
Theorem \ref{thm:11} below.
\begin{main_theorem}
  \label{thm:10}
	For every $p \in (1, \infty)$  there is $C_{p} > 0$ such that for all
	$\big(f_t: t\in\NN\big) \in \ell^p\big(\ell^2\big(\ZZ^{d_0}\big)\big)$ we have
	\[
    \Big\lVert\Big(
	\sum_{t\in\NN}\sup_{N\in\NN}\big|M_N^\calP f_t\big|^2\Big)^{1/2}
	\Big\rVert_{\ell^p}\le
	C_{p}\Big\|\big(\sum_{t\in\NN}|f_t|^2\big)^{1/2}\Big\|_{\ell^p}.
    \]
	Moreover, the constant 
        $C_p$ is independent of  the coefficients of the polynomial
        mapping $\calP$.
\end{main_theorem}
In fact Theorem \ref{thm:0} will be a corollary of the
following vector-valued estimates.
\begin{main_theorem}
  \label{thm:11}
	For every $p \in (1, \infty)$  there is $C_{p} > 0$ such that for all
	$\big(f_t: t\in\NN\big) \in \ell^p\big(\ell^2\big(\ZZ^{d_0}\big)\big)$ we have 
	\[
          	\Big\lVert\Big(
	\sum_{t\in\NN}\sup_{N\in\NN}\big|T_N^\calP f_t\big|^2\Big)^{1/2}
	\Big\rVert_{\ell^p}\le
	C_{p}\Big\|\big(\sum_{t\in\NN}|f_t|^2\big)^{1/2}\Big\|_{\ell^p}.
	\]
	Moreover, the constant 
        $C_p$  is independent of  the coefficients of the polynomial
        mapping $\calP$.
\end{main_theorem}

The proof of Theorem \ref{thm:10} and Theorem \ref{thm:11} will be
based on an analysis of the Fourier multipliers corresponding to
the operators $M_{2^n}^{\calP}$ and $ T_{2^n}^{\calP}$, respectively. In
order to describe the key points of our approach let us focus the
attention on $m_{2^n}$, which is the multiplier associated with
$M_{2^n}^{\calP}$,
i.e. $\calF^{-1}\big(m_{2^n}\hat{f}\big)=M_{2^n}^{\calP}f$.  For simplicity,
let us also assume that we are in the scalar case.  As in the previous papers in the subject we
employ the circle method of Hardy and Littlewood. We use these
techniques implicitly in the  analysis of the  relevant partition
of unity $\Xi_{n}$, see \eqref{eq:9},  which allow us to distinguish between asymptotic or highly
oscillatory behavior of $m_{2^n}$. The projections $\Xi_n$ play an
essential role in our further study of $m_{2^n}$. More precisely, we write
\[
m_{2^n}(\xi)=(1-\Xi_n(\xi))m_{2^n}(\xi)+\Xi_n(\xi)m_{2^n}(\xi),
\]
where $(1-\Xi_n(\xi))m_{2^n}(\xi)$ is a highly oscillatory term and
$\Xi_n(\xi)m_{2^n}(\xi)$ localizes asymptotic behavior of
$m_{2^n}(\xi)$. The highly oscillatory part is controlled by Weyl's
sums, see Section \ref{sec:3}. In the asymptotic part one approximates
the multiplier by sums over irreducible fractions with small
denominators corresponding to the integral analogue of $m_{2^n}(\xi)$
multiplied by Gauss sums.  This part requires a more sophisticated
analysis, and in particular three tools that we now highlight:
\begin{enumerate}[(i)]
\item A variant of a key idea of Ionescu--Wainger \cite{iw}: a suitable
  $\ell^p$ estimate for projection operators which are made up of sums
  corresponding to fractions $a/q$ whose denominators $q$ have
  appropriate limitation in terms of their prime power
  factorization (see \eqref{eq:9} and inequality \eqref{eq:35} below as well
as Theorem \ref{th:3}).
\item A maximal estimate in terms of dyadic sub-blocks (see
  \eqref{eq:21}). It is a consequence of a numerical maximal estimate
  (see Lemma \ref{lem:6}), which in turn is an outgrowth of the idea
  implicit in the proof of the classical Rademacher--Menshov theorem
  (see \cite{kaczste} and also \cite{ll} and \cite{mt3}).
\item A refinement of the estimates for multi-dimensional Weyl's sums in
  \cite{SW0}, where the previous limitations $N^\epsilon \leq q \leq
  N^{k-\epsilon}$ are replaced by the weaker restrictions
  $(\log{N})^\beta \leq q \leq N^k(\log{N})^{-\beta}$ for
  suitable $\beta.$
\end{enumerate}

With these ideas we will show that our unified approach allows us to
deal simultaneously with maximal operators both in the positive and
non-positive cases. It also gives an extension to the vector-valued
form of these results.

We describe now more precisely the outline of this paper and the ingredients of the proofs of
Theorem \ref{thm:10} and Theorem \ref{thm:11}. In Section \ref{sec:2}
we introduce a lifting procedure (see Lemma \ref{lem:1}) which
allows us to replace any polynomial mapping $\calP$ by the canonical
polynomial mapping $\calQ$ which has all coefficients equal to 1. As a result
our bounds will be independent of coefficients of the underlying
polynomial mapping.

Secondly, we formulate Lemma \ref{lem:6}, which is critical in
bounding the supremum norm by square functions.  In \cite{mt3} it was
seen that Lemma \ref{lem:6} can be used as a counterpart or
replacement of Bourgain's logarithmic lemma \cite[Lemma 4.1]{bou},
which was an essential tool in analysis of maximal functions
corresponding to Radon averages. Although a different logarithmic
lemma found many applications in harmonic analysis, especially in
discrete harmonic analysis, it turned out to be limited to averaging
operators. Fortunately Lemma \ref{lem:6} is more flexible and permits us to
circumvent the difficulties involved with non-positive operators. Thus
this lemma is an invaluable ingredient in the unification of operators
of Radon-type, represented by Theorem \ref{thm:10} and Theorem \ref{thm:11} respectively.

We also gather some basic tools which allow us to efficiently compare
discrete $\| \cdot \|_{\ell^p}$ norms with continuous
$\| \cdot \|_{L^p}$ norms. Finally, we formulate Theorem \ref{th:3}
which is a key ingredient in the all steps of our proofs. This theorem
was proven by Ionescu and Wainger in \cite{iw}. Theorem \ref{th:3} is
a deep result and develops the most refined tools to date in the area
of discrete harmonic analysis. The main new idea of Ionescu and
Wainger is to use the technique of strong orthogonality combined ---
in the asymptotic part of the sum --- with a sophisticated
decomposition of the denominators in terms of their prime power
factorization.

In Section \ref{sec:3} we present the variant of multidimensional
Weyl's  sum estimates with logarithmic decay, see Theorem \ref{thm:3}. 
It was known (see \cite{SW0})  for Weyl sums $S_N$ (defined in \eqref{eq:56})  that $|S_N|\le CN^{k-\delta}$
for some $\delta>0$ provided that for at least one coefficient
$\xi_{\gamma_0}$ of a phase polynomial $P$ there are integers $a$ and
$q$ such that  $(a, q)=1$, $|\xi_{\gamma_0}-a/q|\le q^{-2}$ and
$N^{\varepsilon}\le q\le N^{|\gamma_0|-\varepsilon}$ for some $\varepsilon>0$.  
In the sequel proceeding as in \cite{SW0} we will be able to prove that $|S_N|\le CN^k(\log
N)^{-\alpha}$ with arbitrary large $\alpha>0$ provided that $(\log
N)^{\beta}\le q\le N^{|\gamma_0|}(\log N)^{-\beta}$ for some large
$\beta>0$. This logarithmic decay has  great importance for our
further analysis of multipliers in highly oscillatory regime. A one
dimensional variant of Weyl's inequality with logarithmic decay was
known for some time, see for instance \cite{va}
or more recently \cite{w0} and
the references therein.

Sections \ref{sec:4} and Section \ref{sec:7} complete the proofs of Theorems \ref{thm:10} and \ref{thm:11}
respectively. To understand more quickly the structure of the proofs, the reader may begin by looking at
Sections \ref{sec:4} and \ref{sec:7} first. These sections can be read independently, assuming the results
in the previous sections.

In the Appendix \ref{sec:8}, which is self-contained, we collect
vector-valued estimates for the maximal functions of Radon type in the
continuous settings. The proofs of Theorem \ref{thm:17} and Theorem
\ref{thm:18} can be found in \cite{rrt}. However, we provide 
proofs of these results and for the
convenience of the reader we present the details.

\subsection{Notation}
Throughout the whole article, we write $A \lesssim B$ ($A \gtrsim B$) if there is an
absolute constant $C>0$ such that $A\le CB$ ($A\ge CB$).
Moreover, $C > 0$ stand for a large positive constant whose value may vary from
occurrence to occurrence. If $A \lesssim B$ and $A\gtrsim B$ hold simultaneously then we write
$A \simeq B$. We also write $A \lesssim_{\delta} B$ ($A \gtrsim_{\delta} B$) to
indicate that the constant $C>0$ depends on some $\delta > 0$. Let $\NN_0 = \NN \cup
\{0\}$. For $N \in \NN$ we set
\[
    \NN_N = \{1, 2, \ldots, N\}, \quad \text{and} \quad
    \ZZ_N = \{-N, \ldots, -1, 0, 1, \ldots, N\}.
\]
For a vector $x \in \RR^d$ we use the following norms
\[
    \norm{x}_\infty = \max\{\abs{x_j} : 1 \leq j \leq d\}, \quad \text{and} \quad
    \norm{x} = \Big(\sum_{j = 1}^d \abs{x_j}^2\Big)^{1/2}.
\]
If $\gamma\in \NN_0^k$ is a multi-index  then $\norm{\gamma} = \gamma_1 + \ldots
+ \gamma_k$. Although, we use $|\cdot|$ for the length of a multi-index $\gamma\in \NN_0^k$
and the Euclidean norm of $x\in\RR^d$, their meaning is always clear from the context.
Finally, let $\mathcal{D}=\{2^n: n\in\NN_0\}$ denote the set of dyadic numbers.

\subsection*{Acknowledgments}
The authors thank the Hausdorff Research Institute for Mathematics in Bonn for
hospitality in 2014 during the  program ``Harmonic Analysis and
Partial Differential Equations''. 
\section{Preliminaries}
\label{sec:2}
We begin this section by establishing a result that extends the
Marcinkiewicz--Zygmund inequality to the Hilbert space setting. Let
$(X, \mu)$ be a measure space and fix $p, r\in(0, \infty]$. We say
that a sequence of complex-valued functions
$(f_t: t\in \NN_0)\in L^p\big(\ell^r(X)\big)$ if
\[
\Big\|\big(\sum_{t\in\NN_0}|f_t|^r\big)^{1/r}\Big\|_{L^p}<\infty
\]
with obvious modifications when $p=\infty$ or $r=\infty$.  In our case
$(X, \mu)$ is usually $\RR^d$ with the Lebesgue measure or $\ZZ^d$ with the counting measure.
      Let $\big(T_m: m\in\NN_0\big)$ be a family of bounded linear operators
$T_m:L^p(X)\to L^p(X)$.  Moreover, for each $\omega\in[0, 1]$ we
define
\[
T^{\omega}=\sum_{m\in\NN_0}\varepsilon_m(\omega)T_m,
\]
where $\big(\varepsilon_m: m\in\NN_0\big)$ is the sequence of
 Rademacher functions on  $[0,1]$.
\begin{lemma}
  \label{lem:30}
  Suppose that for every $p\in(0, \infty)$ there is a constant $\mathbf C_p>0$ such that
for all $\omega\in[0, 1]$ and $f\in L^p(X)$ we have
\begin{align}
  \label{eq:209}
  \big\|T^{\omega}f\big\|_{L^p}\le \mathbf C_p\|f\|_{L^p}.
\end{align}
Then there is a constant  $C>0$ such that  for every sequence $\big(f_t:
t\in\NN_0\big)\in L^p\big(\ell^2(X)\big)$ we have
\begin{align}
\label{eq:210}
  \Big\|\big(\sum_{t\in\NN_0}\sum_{m\in\NN_0}|T_mf_t|^2\big)^{1/2}\Big\|_{L^p}\le
  C\mathbf C_p
  \Big\|\big(\sum_{t\in\NN_0}|f_t|^2\big)^{1/2}\Big\|_{L^p}.
\end{align}
In particular, if  $T_m\equiv0$ for each $m\in\NN$ then \eqref{eq:210}
implies  the Marcinkiewicz--Zygmund result
\begin{align}
\label{eq:69}
  \Big\|\big(\sum_{t\in\NN_0}|T_0f_t|^2\big)^{1/2}\Big\|_{L^p}\le
  C\mathbf C_{p}
  \Big\|\big(\sum_{t\in\NN_0}|f_t|^2\big)^{1/2}\Big\|_{L^p}.
\end{align}
Summation over $\NN_0$ in the inner sum  of \eqref{eq:210} can be
replaced by any other countable set and the result  remains valid.
\end{lemma}
\begin{proof}
  Let us define
\[
F_{\omega'}=\sum_{t\in\NN_0}\varepsilon_t(\omega')f_t.
\]
Then by the several variable variant of Khinchine's inequality \cite[Appendix D.3]{sings} we obtain
 \begin{align*}
   \Big\|\big(\sum_{t\in\NN_0}\sum_{m\in\NN_0}|T_mf_t|^2\big)^{1/2}\Big\|_{L^p}^p
	\lesssim
	\int_0^1\int_0^1\big\|T^{\omega}(F_{\omega'})\big\|_{L^p}^p
	{\:\rm d}\omega{\:\rm d}\omega'.
 \end{align*}
By \eqref{eq:209}, for each $\omega'\in[0, 1]$,
\begin{align*}
  \big\|T^{\omega}(F_{\omega'})\big\|_{L^p}\lesssim \|F_{\omega'}\|_{L^p},
\end{align*}
thus
\[
	\Big\|\big(\sum_{t\in\NN_0}\sum_{m\in\NN_0}|T_mf_t|^2\big)^{1/2}\Big\|_{L^p}^p
    \lesssim
	\int_0^1\|F_{\omega'}\|_{L^p}^p
    {\:\rm d}\omega',
\]
and since by another application of Khinchine's inequality
\begin{align*}
  \int_0^1\|F_{\omega'}\|_{L^p}^p{\: \rm d}\omega'\lesssim 
  \Big\|\big(\sum_{t\in\NN_0}|f_t|^2\big)^{1/2}\Big\|_{L^p}^p,
\end{align*}
the proof of the lemma is complete. 
\end{proof}

\subsection{Lifting lemma}
\label{sec:5}
Let $\calP=(\calP_1,\ldots, \calP_{d_0}): \ZZ^k \rightarrow \ZZ^{d_0}$
be a polynomial mapping whose
components $\calP_j$ are  polynomials on $\ZZ^k$ with integer coefficients such that $\calP_j(0) = 0$. We set
$$
\deg \calP = \max\{ \deg \calP_j : 1 \leq j \leq d_0\}.
$$
It is convenient to work with the set
$$
\Gamma =
\big\{
	\gamma \in \ZZ^k \setminus\{0\} : 0 \leq |\gamma| \leq \deg \calP
\big\}
$$
with the lexicographic order. Observe that for every $j\in\{1, \ldots, d_0\}$ there are coefficients
$(c_j^\gamma: \gamma\in \Gamma )\subset \ZZ$ such that each $\calP_j$ can be expressed as
$$
\calP_j(x) = \sum_{\gamma \in \Gamma} c_j^\gamma x^\gamma.
$$
 Let us denote by $d$ the cardinality of the set $\Gamma$.
We identify $\RR^d$ with the space of all vectors whose coordinates are labeled by multi-indices
$\gamma \in \Gamma$. Let $A$ be a diagonal $d \times d$ matrix such that
\begin{equation}
	\label{eq:25}
	(A v)_\gamma = \abs{\gamma} v_\gamma.
\end{equation}
For $t > 0$ we set
$$
t^{A}=\exp(A\log t),
$$
i.e., $t^A x=\big(t^{|\gamma|}x_{\gamma}: \gamma\in \Gamma\big)$ for
every $x\in\RR^d$. Next, we introduce the \emph{canonical} polynomial mapping
$$
\calQ = \big(\seq{\calQ_\gamma}{\gamma \in \Gamma}\big) : \ZZ^k \rightarrow \ZZ^d
$$
where $\calQ_\gamma(x) = x^\gamma$ and $x^\gamma=x_1^{\gamma_1}\cdot\ldots\cdot x_k^{\gamma_k}$.
The coefficients $\big(\seq{c_j^\gamma}{\gamma \in \Gamma, j \in \{1, \ldots, d_0\}}\big)$ define
a linear transformation $L: \RR^d \rightarrow \RR^{d_0}$ such that $L\calQ = \calP$. Indeed, it suffices to set
$$
(L v)_j = \sum_{\gamma \in \Gamma} c_j^\gamma v_\gamma
$$
for each $j \in \{1, \ldots, d_0\}$ and $v \in \RR^d$. 

For the future reference Lemma \ref{lem:1} is stated in a more general form than needed. To do so,
let us denote by $\calN$ a seminorm defined on sequences of complex numbers, that is a non-negative
function such that for any two sequences $\big(\seq{a_j}{j \in J}\big)$ and $\big(\seq{b_j}{j \in J}\big)$
where $J \subseteq \ZZ$, satisfies
\[
	\calN\big(\seq{a_j + b_j}{j \in J}\big) 
	\leq 
	\calN\big(\seq{a_j}{j \in J}\big) 
	+ \calN\big(\seq{b_j}{j \in J}\big),
\]
and for any $\lambda \in \CC$,
\[
	\calN\big(\seq{\lambda a_j}{j \in J}\big) = \abs{\lambda} \calN\big(\seq{a_j}{j \in J}\big).
\]
We also assume that
\[
  \calN\big(\seq{a_j}{j \in J}\big)\le \Big(\sum_{j\in J}|a_j|^2\Big)^{1/2},
\]
and 
\[
\calN\big(\seq{a_j}{j \in J_1}\big)\le \calN\big(\seq{a_j}{j \in J_2}\big).
\]
for any $J_1\subseteq J_2\subseteq \ZZ$. For the first reading one may think that
\[
	\calN\big(\seq{a_j}{j \in J}\big) = \sup_{j \in J} \abs{a_j}.
\]
The next lemma, inspired by the continuous analogue (see \cite{deL} or \cite[Chapter 11]{bigs}), where the latter reduces
the proof of Theorem \ref{thm:0} to the canonical polynomial mapping.
\begin{lemma}
	\label{lem:1} Let $R_N^{\calP}$ be either $M_N^{\calP}$ or $T_N^{\calP}$.
	Suppose that for some $p, r \in (0, \infty]$ there is a
        constant $C=C_{p, r} > 0$ such that
	\begin{equation}
		\label{eq:162}
		\Big\lVert\Big(\sum_{t\in\NN}
		\mathcal N\big(R_N^\calQ f_t: N \in \NN\big)^r\Big)^{1/r}
		\Big\rVert_{\ell^p(\ZZ^d)}
		\leq C\Big\|\Big(\sum_{t\in\NN}|f_t|^r\Big)^{1/r}\Big\|_{\ell^p(\ZZ^d)}.
	\end{equation}
	Then
	\begin{equation}
		\label{eq:32}
		\Big\lVert\Big(\sum_{t\in\NN}
		\mathcal N\big(R_N^\calP f_t: N \in \NN\big)^r\Big)^{1/r}
		\Big\rVert_{\ell^p(\ZZ^{d_0})}
		\leq C\Big\|\Big(\sum_{t\in\NN}|f_t|^r\Big)^{1/r}\Big\|_{\ell^p(\ZZ^{d_0})}.
	\end{equation}
\end{lemma}
\begin{proof}
	Let $M > 0$ and $\Lambda > 0$ be fixed. In the proof we
	let $x \in \ZZ^{d_0}$, $y \in \ZZ^k$ and $u \in \ZZ^d$. For any $x \in \ZZ^{d_0}$ we define a
	function $F^x_t$ on $\ZZ^d$ by
	$$
	F^x_t(z) =
	\begin{cases}
		f_t(x + L(z)) & \text{ if } \norm{z}_{\infty} \leq M + \Lambda^{k N_0},\\
		0 & \text{ otherwise.}
	\end{cases}
	$$
	If $\norm{y}_{\infty} \leq N\le \Lambda$ and $\norm{u}_{\infty} \leq M$ then
	$\norm{u - \calQ(y)}_{\infty} \leq M + \Lambda^{kN_0}$.
	Therefore, for each $x \in \ZZ^{d_0}$
	$$
	R_N^\calP f_t(x+L u)
	=
	\sum_{y \in \ZZ_N^k\setminus\{0\}} f_t\big(x + L\big(u - \calQ(y)\big)\big)H_N(y)
	=
	R_N^\calQ F^x_t(u),
	$$
    where 
	\[
		H_N(y)=
		\begin{cases}
			N^{-k}\ind{[1, N]^k}(y) & \text{ if } R_N^\calP=M_N^\calP, \\
			\ind{[-N, N]^k\setminus\{0\}}(y)K(y) & \text{ if } R_N^\calP=T_N^\calP.
		\end{cases}
	\]
	Hence,
	\begin{align*}
		& \Big\lVert\Big(\sum_{t\in\NN}
		\mathcal N\big(R_N^\calP f_t: N \in [1, \Lambda]\big)^r\Big)^{1/r}
		\Big\rVert_{\ell^p(\ZZ^{d_0})}^p \\
		& \qquad\qquad =
		\frac{1}{(2M+1)^{d}}
		\sum_{x \in \ZZ^{d_0}}
		\sum_{\norm{u}_\infty \leq M}
\Big(\sum_{t\in\NN}
		\mathcal N\big(R_N^\calP f_t(x+Lu): N \in [1, \Lambda]\big)^r\Big)^{p/r}\\
		& \qquad \qquad=
		\frac{1}{(2M+1)^{d}}
		\sum_{x \in \ZZ^{d_0}}
		\sum_{\norm{u}_\infty \leq M}
\Big(\sum_{t\in\NN}
		\mathcal N\big(R_N^\calQ F^x_t(u): N \in [1, \Lambda]\big)^r\Big)^{p/r}\\
		& \qquad \qquad \leq
		\frac{C^p}{(2M+1)^d}
		\sum_{x \in \ZZ^{d_0}}
		\sum_{u \in \ZZ^d}\Big(\sum_{t\in\NN}
		 |F^x_t(u)|^r\Big)^{p/r},
	\end{align*}
	where in the last inequality we have used \eqref{eq:162}. Since
    \begin{align*}
		\sum_{x \in \ZZ^{d_0}}
		\sum_{u \in \ZZ^d}\Big(\sum_{t\in\NN}
		 |F^x_t(u)|^r\Big)^{p/r}
		& =
		\sum_{x \in \ZZ^{d_0}} \sum_{\norm{u}_\infty \leq M + \Lambda^{kN_0}}
		\Big(\sum_{t\in\NN}
		 |f_t(x + L u)|^r\Big)^{p/r}\\
		& \leq \big(2M + 2\Lambda^{kN_0}+1\big)^d 
		\Big\|\Big(\sum_{t\in\NN}|f_t|^r\Big)^{1/r}\Big\|_{\ell^p(\ZZ^{d_0})}^p,
    \end{align*}
	we get
	\[
		\Big\lVert\Big(\sum_{t\in\NN}
		\mathcal N\big(R_N^\calP f_t: N \in [1, \Lambda]\big)^r\Big)^{1/r}
		\Big\rVert_{\ell^p(\ZZ^{d_0})}^p\leq
	C^p
	\bigg(1 + \frac{\Lambda^{kN_0}}{M}\bigg)^d
	\Big\|\Big(\sum_{t\in\NN}|f_t|^r\Big)^{1/r}\Big\|_{\ell^p(\ZZ^{d_0})}^p.
	\]
	Taking $M\to\infty$, we conclude that
	$$
	\Big\lVert\Big(\sum_{t\in\NN}
	\mathcal N\big(R_N^\calP f_t: N \in [1, \Lambda]\big)^r\Big)^{1/r}
	\Big\rVert_{\ell^p(\ZZ^{d_0})}^p
	\leq
	C^p
	\Big\|\Big(\sum_{t\in\NN}|f_t|^r\Big)^{1/r}\Big\|_{\ell^p(\ZZ^{d_0})}^p,
	$$
	which by the monotone convergence theorem implies \eqref{eq:32}.
\end{proof}
In the rest of the article $M_N = M_N^\calQ$ and $T_N = T_N^\calQ$ denote the operators defined for the canonical
polynomial mapping $\calQ$.

\subsection{Basic numerical inequality}
Here we remind the reader of the following simple observation which is
essential in the sequel, see Section \ref{sec:4} and Section
\ref{sec:7}.
\begin{lemma}
	\label{lem:6}
	For any sequence $(\seq{a_j}{0 \leq j \leq 2^s})$ of complex numbers we have
	\begin{equation}
		\label{eq:21}
		\max_{0 \leq j \leq 2^s} \abs{a_j}
		\leq
		\abs{a_{j_0}}
		+
		\sqrt{2}
		\sum_{i = 0}^s
		\Big(
		\sum_{j = 0}^{2^{s-i}-1}
		\abs{a_{(j+1)2^i} - a_{j 2^i}}^2
		\Big)^{1/2}
	\end{equation}
	for any $j_0 \in \{0, 1, \ldots, 2^s\}$.
\end{lemma}
A variational variant of this inequality was proven by Lewko--Lewko \cite[Lemma 13]{ll} in the context of
variational Rademacher--Menshov theorems. It was also obtained independently by the first two authors in
\cite{mt3} in the context of variational estimates for discrete Radon transforms, see also \cite{mst2}.

\subsection{Sampling principle}
Let $\calF$ denote the Fourier transform on $\RR^d$ defined for any 
$f \in L^1\big(\RR^d\big)$ as
$$
\calF f(\xi) = \int_{\RR^d} f(x) e^{2\pi i \sprod{\xi}{x}} {\: \rm d}x, \quad \text{for} \quad \xi\in\RR^d.
$$
If $f \in \ell^1\big(\ZZ^d\big)$, we set
$$
\hat{f}(\xi) = \sum_{x \in \ZZ^d} f(x) e^{2\pi i \sprod{\xi}{x}}, \quad \text{for} \quad \xi\in\TT^d.
$$
To simplify notation we denote by $\mathcal F^{-1}$ the inverse Fourier transform on $\RR^d$
or the inverse Fourier transform on the torus $\TT^d$ (Fourier coefficients), depending on the context.

Let $\big(\seq{\Theta_N}{N \in \NN}\big)$ be a sequence of multipliers on $\RR^d$
with the property that for each $p \in (1, \infty)$ there is a constant $\boldB_{p} > 0$
such that for any $\big(f_t: t\in\ZZ\big) \in
L^p\big(\ell^2\big(\RR^d\big)\big) \cap L^2\big(\ell^2\big(\RR^d\big)\big)$ we have
\begin{equation}
	\label{eq:54}
	\Big\lVert\Big(\sum_{t\in\ZZ}
	\calN\big(\seq{\calF^{-1}\big(\Theta_N \calF f_t \big)}{N \in
          \NN}\big)^2\Big)^{1/2}
	\Big\rVert_{L^p}
	\leq
	\boldB_{p}
	\Big\|\Big(\sum_{t\in\ZZ}|f_t|^2\Big)^{1/2}\Big\|_{L^p},
\end{equation}
where $\calN$ is a seminorm defined for sequences of complex numbers as in Section \ref{sec:5}.

We state a discrete analogue of \eqref{eq:54}. For this
purpose, let $\eta: \RR^d \rightarrow \RR$ be a smooth function such
that $0 \leq \eta(x) \leq 1$, and
\[
\eta(x) =
\begin{cases}
	1 & \text{ for } \norm{x} \leq 1/(16 d),\\
	0 & \text{ for } \norm{x} \geq 1/(8 d).
\end{cases}
\]
Additionally, we assume that $\eta$ is a convolution of two
non-negative smooth functions $\phi$ and $\psi$ with compact supports
contained inside $(-1/(8d), 1/(8d))^d$.  This is only a technical
assumption which allows us to adopt some arguments from
\cite{mt3}.  From now on, unless otherwise stated, we assume that
every function $f_t: \ZZ^d\to\CC$ is finitely supported.

\begin{proposition}
\label{prop:4}
 Suppose that \eqref{eq:54}
holds for some $p\in (1, \infty)$. Then there is a constant $C > 0$
such that for each $Q\in\NN$ and $m \in \NN_Q^k$ and any 
 diagonal $d\times d$ matrix $R$ with positive entries
$(r_{\gamma}: \gamma\in\Gamma)$ satisfying 
$\inf_{\gamma\in\Gamma}r_{\gamma}\geq 2^{2d+2} Q^{d+1}$ 
 we have
 \begin{align}
   \begin{split}
   \label{eq:1}
		\Big\lVert\Big(\sum_{t\in\ZZ}
		\calN
		\big(\seq{\calF^{-1}\big(\Theta_N\eta(R \: \cdot \:) 
		\hat{f_t}\big)(Q x + m)}{N \in \NN}\big)^2&\Big)^{1/2}
		\Big\rVert_{\ell^p(x)}\\
		\leq
		C \boldB_p
		\Big\lVert\Big(\sum_{t\in\ZZ}
		\big|\calF^{-1}\big(\eta(R \: \cdot \:) \hat{f_t} \big)&(Q x + m)\big|^2\Big)^{1/2}
		\Big\rVert_{\ell^p(x)}.
   \end{split}
                \end{align}
\end{proposition}

Proposition \ref{prop:4} for $r$-variational seminoms in the
scalar-valued case was proven in \cite[Proposition 3.3]{mt3}. The same
methods can be adapted to deduce inequality \eqref{eq:1}, therefore
we omit the proof. See also the discussion of sampling in
\cite[Proposition 2.1 and Corollary 2.5]{MSW}, which provides an earlier approach.

\subsection{Estimates for Ionescu--Wainger type multipliers}
We first introduce necessary notation. Let $\rho>0$. For every $N\in\NN$ we define
\[
	N_0=\lfloor
	N^{\rho/2}\rfloor+1 \qquad \text{and} \qquad Q_0=(N_0!)^D,
\]
where $D=D_{\rho}=\lfloor 2/\rho\rfloor+1$.  Let $\mathbb P$ denote the set
of all prime numbers and $\mathbb P_N=\mathbb P\cap (N_0, N]$. For any
$V\subseteq\mathbb P_N$ and $k\in\NN_D$, let 
\[
\Pi_k(V)=\big\{p_{1}^{\gamma_{1}}\cdot\ldots\cdot
p_{k}^{\gamma_{k}}: \ \gamma_{l}\in\NN_D\  \text{and}\ p_{l}\in
V\ \text{are distinct for all $1\le l\le k$} \big\}.
\]
Then $\Pi_{k_1}(V)\cap \Pi_{k_2}(V)=\emptyset$ whenever $k_1\not=k_2$. Let
\[
\Pi(V)=\bigcup_{k\in\NN_D}\Pi_k(V)
\]
be the set of all products of primes factors from $V$ of
length at most $D$, at powers between $1$ and $D$. 
Next, we introduce the sets
\[
 P_N=\big\{q=Q\cdot w: Q|Q_0\ \text{and}\ w\in
\Pi(\mathbb P_{N})\cup\{1\}\big\}.
\]
It is not difficult to see that:
\begin{itemize}
\item every integer $q\in\NN_N$ can be uniquely
written as $q=Q\cdot w$, where $Q|Q_0$ and $w\in \Pi(\mathbb
P_N)\cup\{1\}$;
\item there is $C_{\rho}>0$ such that for every $N\ge C_{\rho}$ 
 we obtain
\[q=Q\cdot w\le Q_0\cdot w\le (N_0!)^DN^{D^2}\le
e^{N^{\rho}};
\]
\item 
$\NN_N\subseteq P_N\subseteq\NN_{e^{N^{\rho}}}$;
\item   $P_{N_1}\subseteq P_{N_2}$, if
$N_1\le N_2$.
\end{itemize}
For a subset $S\subseteq\NN$ we define 
\[
\mathcal R(S)=\big\{a/q \in \TT^d\cap\QQ^d : a \in A_q \text{ and }
q\in S\big\},
\] 
where for each $q\in\NN$ 
\[
A_q=\big\{a\in\NN_q^d: \gcd\big(q, (a_{\gamma}: \gamma\in\Gamma)\big)=1\big\}.
\]
Finally, for each $N\in\NN$ we set
\begin{align}
  \label{eq:156}
\mathscr{U}_N=\mathcal R(P_N).
\end{align}
If $N_1\le N_2$ then  $\mathscr{U}_{N_1}\subseteq \mathscr{U}_{N_2}$ and we have the estimate
\begin{align}
  \label{eq:2}
  |\mathscr{U}_N|\lesssim e^{(d+1)N^{\rho}}.
\end{align}

We are now in the position to define Ionescu--Wainger type multipliers. Assume that $\Theta$
is a multiplier on $\RR^d$ and for every
$p\in(1, \infty)$ there is a constant $\boldA_p >0 $ such that for
every $f\in L^2\big(\RR^d\big)\cap L^p\big(\RR^d\big)$ we have
\[
  \big\lVert\calF^{-1}\big(\Theta\calF f\big)\big\rVert_{L^p}
  \leq
  \boldA_p \vnorm{f}_{L^p}.
\]
For each $N \in \NN$ and $\xi\in\TT^d$
we define a new periodic  multiplier
\begin{align*}
	\Delta_N(\xi)
	=\sum_{a/q \in\mathscr{U}_N}
	 \Theta(\xi - a/q) \eta_N(\xi - a/q),
\end{align*}
where $\eta_N(\xi)=\eta\big(\mathcal E_N^{-1}\xi\big)$ and $\mathcal E_N$
is a diagonal $d\times d$  matrix with  positive entries
$(\varepsilon_{\gamma}: \gamma\in\Gamma)$ such that
$\varepsilon_{\gamma}
\le e^{-N^{2\rho}}$.
\begin{theorem}
	\label{th:3}
	For every $\rho>0$ and $p\in(1, \infty)$ there is a
    constant $C_{\rho, p} > 0$ such that
	for any $N\in\NN$ and $f \in \ell^p\big(\ZZ^d\big)$ we have
	\begin{align}
          \label{eq:5}
		\big\lVert
		\calF^{-1}\big(\Delta_{N}
		\hat{f}\big)\big\rVert_{\ell^p} 
		\leq C_{\rho, p} \boldA_{2r}(\log N)
		\vnorm{f}_{\ell^p},
	\end{align}
	where $r = \max\big\{\lceil p/2 \rceil, \lceil p'/2 \rceil\big\}$.
\end{theorem}
Theorem \ref{th:3} was established by Ionescu and Wainger in \cite{iw}
with $(\log N)^{2/\rho}$ in place of $\log N$. Replacing \cite[Lemma
3.1]{iw} with \cite[Lemma 2.4]{mir1} and arguing as in \cite{iw} one
obtains inequality \eqref{eq:5}.  For a slightly different approach
we refer to \cite[Section 2]{mir1}. Theorem \ref{th:3} is one of the
key ingredients in all of the steps of our proofs. 

\section{Exponential sums}
\label{sec:3}
This section is devoted to study certain exponential sums.  We first fix some notation. Let $B_r(x_0)$ denote an Euclidean ball in $\RR^k$ centered at $x_0$ with radius $r>0$.
Let $P$ be a polynomial with real coefficients in $\RR^k$ of degree $d \in\NN$ such that
\begin{align}
  \label{eq:12}
	P(x)  = \sum_{\gamma\in\NN_0^k\colon 0 < \norm{\gamma} \leq d} \xi_\gamma x^\gamma,\quad \text{and}\quad P(0)=0.  
\end{align}
Let $\varphi:\RR^k \rightarrow \CC$  be a  function in $\mathcal
C^1\big(\RR^k\big)$ satisfying the following conditions 
\begin{align}
  \label{eq:217}
|\varphi(x)|\le 1, \qquad \text{and} \qquad |\nabla \varphi(x)|\le (1+|x|)^{-1}.  
\end{align}
Our aim is to show Theorem \ref{thm:3}, which is a refinement of \cite[Proposition 3]{SW0}.
\begin{theorem}
	\label{thm:3}
        For every $d, k\in\NN$ and $\alpha>0$ there are
        $\beta_{\alpha}=\beta_{\alpha}(d, k)\ge\alpha$ and $C>0$ such that
        for every $\beta\ge\beta_{\alpha}$, every  $N\ge1$, every polynomial $P$ as in \eqref{eq:12},
         and every convex set
        $\Omega\subseteq B_N(0)$ the following holds. Suppose that for some  multi-index $\gamma_0\in\NN_0^k$
		so that $0 < \norm{\gamma_0} \leq d$, there are  integers $0\le a\le q$ with $(a, q) = 1$, and
        \begin{align}
          \label{eq:15}
        	\Big\lvert
		\xi_{\gamma_0} - \frac{a}{q}
		\Big\rvert
		\leq
		\frac{1}{q^2},  
        \end{align}
                and
		\[
		(\log N)^\beta \leq q \leq N^{\norm{\gamma_0}} (\log N)^{-\beta}.
		\]
        Then 
	\begin{equation}
		\label{eq:56}
		\Big|\sum_{n \in \Omega \cap \ZZ^k} e^{2\pi i P(n)}\varphi(n)\Big|
		\leq
		C
		N^k (\log N)^{-\alpha}.
	\end{equation}
	The implied  constant $C$ may depend on $d, k, \alpha$ and the function $\varphi$ from \eqref{eq:217}, but is independent of $a, q, N$ and the coefficients of $P$ and the sets $\Omega$. 
\end{theorem}

In order to prove Theorem \ref{thm:3} we will argue by the backward
induction on $|\gamma_0|$.  This approach forces us to consider the
summation over convex sets $\Omega$ in the exponential sums.  Even
when one starts with sums defined over a cube, certain changes of
variables are required in the argument.  Then the cube is transformed
under a linear integral map and the best what one can say about
this image is that the resulting set is convex. On the other hand, the exponential sums with
summation over convex sets allow us to consider multipliers corresponding to the discrete averaging Radon operators defined over convex sets, see for instance \cite{mst2}.
We begin with the following simple observation.

\begin{proposition}
  \label{prop:201}
Fix $0\le \sigma\le 1/3$ and assume that $\Omega \subseteq\RR^k$ is a convex set contained in a
ball with radius $r\ge1$. Let $N_{\Omega}=\#\{x\in\Omega\cap\ZZ^k: \dist(x, \partial \Omega)< s\}$.
\begin{enumerate}
\item If $1\le s\le r^{1-3\sigma}$, then
$N_{\Omega}=\mathcal O\big(r^{k-\sigma}\big)$.
\item If $1\le s\le
r$ and $\Omega $ contains a ball
$B_{cr}(x_0')$ for some $x_0'\in\RR^k$ and $c>0$, then $N_{\Omega}=\mathcal O\big(sr^{k-1}\big)$.
\end{enumerate}
\end{proposition}
\begin{proof}
In both cases, we assume that $r$ is large, otherwise the
assertions easily follow.  If $|\Omega|\le r^{k-\sigma}$, then there
is nothing to do since by Davenport's result (see \cite{daven}) we know that
\[
\#\big(\Omega\cap\ZZ^k\big)=|\Omega|+\mathcal O\big(r^{k-1}\big),
\]
where $|\Omega|$ is the volume of $\Omega$. Assume now that
$|\Omega|\ge r^{k-\sigma}$ and we shall show that $N_{\Omega}=\mathcal O\big(sr^{k-1+2\sigma}\big)$,
which gives the desired conclusion if $1\le s\le r^{1-3\sigma}$.  By a simple
integration argument for each $1\le j\le k$ there is a segment
$I_j\subseteq \Omega$ in direction $x_j$, and
$|I_j|\gtrsim\frac{|\Omega|}{r^{k-1}}$. Thus there is a ball $B_{\rho}(x_0')$ 
with radius $\rho=\frac{c|\Omega|}{r^{k-1}}$, for some $c>0$, contained in the
 convex set $\Pi$ generated by all segments $I_1,\ldots, I_k$ with
 $x_0'\in\ZZ^k$ being the closest point to the barycenter of $\Pi$.
It is always possible to choose a ball $B_{\rho}(x_0')$ centered at $x_0'\in\ZZ^k$, by taking $c>0$ very small in the definition of $\rho$, since  $\rho\ge cr^{1-\sigma}$
and we have assumed that $r$ is large. We will also assume that
$x_0'=0$, since  the number of lattice
points is invariant under translations.

   For any
  $x\in\Omega$, let $\bar x\in\partial\Omega$ so that
  $x=\lambda_x\bar x$ with some $\lambda_x\in(0, 1)$. If $\Gamma_{\bar x}$
  is the convex set generated by $\bar x$ and $B_{\rho}(0)$, then the
  angle of the aperture at the vertex $\bar x$ is $\ge\alpha$, with
  some $\alpha\ge\rho/r$ uniformly for $x\in\Omega$. Then
there exists $c'>0$ such that  for all $x\in\Omega$ we get
\[
\dist(x, \partial \Omega)\ge\dist(x, \partial \Gamma_{\bar x})=(1-\lambda_x)|\bar
x|\sin(\alpha/2)\ge c'(1-\lambda_x)\frac{\rho^2}{r}.
\]  

Define $\delta=1-\frac{sr}{c'\rho^2}$ and observe that 
for every  $x\in\Omega_{\delta}=\{y\in\RR^k: \delta^{-1}y\in\Omega\}$
we have
\[
\dist(x, \partial \Omega)\ge c'(1-\lambda_x)\frac{\rho^2}{r}\ge c'(1-\delta)\frac{\rho^2}{r}=s,
\]
which is equivalent to
\[
\{x\in\Omega: \dist(x, \partial \Omega)< s\}\subseteq \Omega\setminus\Omega_{\delta}.
\]
Since $\#\big(\Omega\cap\ZZ^k\big)=|\Omega|+\mathcal
O(r^{k-1})$ and  $\#\big(\Omega_{\delta}\cap\ZZ^k\big)=\delta^k|\Omega|+\mathcal
O(r^{k-1})$ we obtain
\begin{align*}
	N_{\Omega}&=\#\{x\in\Omega\cap\ZZ^k: \dist(x, \partial \Omega)< s\} 
	 \le \#\big(\Omega\cap\ZZ^k\big)-\#\big(\Omega_{\delta}\cap\ZZ^k\big) \\
	 &= \mathcal O\big((1-\delta)r^k\big)+\mathcal O\big(r^{k-1}\big)=\mathcal O\big(sr^{k-1+2\sigma}\big)+\mathcal O\big(r^{k-1}\big).
\end{align*}

For the second part suppose that $B_{cr}(x_0')\subseteq \Omega\subseteq B_r(x_0)$ for
some $x_0, x_0'\in\RR^k$
and $c>0$. Again without  loss of generality we may assume that $x_0'=0$, and $B_{cr}(x_0')$ and
  $\Omega$ are not tangent at any point, otherwise it suffices to take
  $c/2$ instead of $c$. In a similar way as above, for any
  $x\in\Omega$, let $\bar x\in\partial\Omega$ so that
  $x=\lambda_x\bar x$ with some $\lambda_x\in(0, 1)$. We now set $\Gamma_{\bar x}$ to be
  the convex set generated by $\bar x$ and $B_{cr}(0)$, then the
  angle of the aperture at the vertex $\bar x$ is $\ge\alpha$, with
  some $\alpha=\alpha(c)>0$ uniformly for $x\in\Omega$. Then
there exists $c'>0$ such that  for all $x\in\Omega$ we get
\[
\dist(x, \partial \Omega)\ge\dist(x, \partial \Gamma_{\bar x})=(1-\lambda_x)|\bar
x|\sin(\alpha/2)\ge c'r(1-\lambda_x).
\]  
Define $\delta=1-\frac{s}{c'r}$ and observe that 
for every  $x\in\Omega_{\delta}=\{y\in\RR^k: \delta^{-1}y\in\Omega\}$
we have
\[
\dist(x, \partial \Omega)\ge c'r(1-\lambda_x)\ge c'r(1-\delta)=s.
\]
Thus $\{x\in\Omega: \dist(x, \partial \Omega)< s\}\subseteq
\Omega\setminus\Omega_{\delta}$ and  consequently we
conclude $N_{\Omega}=\mathcal O\big(sr^{k-1}\big)$
as desired.
\end{proof}

\begin{remark}
This proposition is needed as a replacement for   \cite[Proposition 9]{SW0}, 
since the latter proposition contains an error. While the present 
version is weaker than
the one in \cite{SW0}, it is sufficiently strong for our purposes, and also to 
establish the Weyl's sum estimates in \cite{SW0}.  
\end{remark}

\begin{lemma} [cf. {\cite[Lemma 1]{SW0}}]
		\label{lem:11}
		Suppose that
		\[
			\Big\lvert
			\theta 
			- 
			\frac{a}{q}
			\Big\rvert
			\leq
			\frac{1}{q^2}
		\]
		and $(\log N)^\beta \leq q \leq N^{l} (\log N)^{-\beta}$ for some $l\in\NN$. Let $Q$ be an integer, $Q \leq (\log N)^{\beta'}$
		with $\beta' < \beta$. If
		\[ 
			\beta_2 \leq \min\{\beta/2, \beta - \beta'\},
		\]
		then there are integers $0\le \tilde{a}\le \tilde{q}$, so that  $(\tilde{a}, \tilde{q}) = 1$ and
		\[
			\Big\lvert
			Q \theta - \frac{\tilde{a}}{\tilde{q}}
			\Big\rvert
			\leq
			\frac{(\log N)^{\beta_2}}{\tilde{q}N^l}
		\]
		with $\frac{1}{2}(\log N)^{\beta_2} \leq \tilde{q} \leq N^{l}(\log N)^{-\beta_2}$.
	\end{lemma}
	\begin{proof}
        We apply Dirichlet's theorem to $Q \theta$ and obtain
                integers $0\le \tilde{a}\le \tilde{q}$, so that  $(\tilde{a}, \tilde{q}) = 1$, with
		\begin{equation}
			\label{eq:182}
			\Big\lvert 
			Q \theta - \frac{\tilde{a}}{\tilde{q}}
			\Big\rvert
			\leq
			\frac{(\log N)^{\beta_2}}{\tilde{q}N^{l}}  
		\end{equation}
		and $\tilde{q} \leq N^{l} (\log N)^{-\beta_2}$. We must see
                that $(\log N)^{\beta_2} \leq \tilde{q}$. There are two
                cases. First, $\tilde{a}/\tilde{q} = Q a/q$, then $\tilde{q} \geq q/Q \geq (\log N)^{\beta - \beta'} \geq (\log N)^{\beta_2}$, and we
		are done. Second, we suppose $\tilde{a}/\tilde{q} \neq Q a /q$. Then
		\[
			\frac{1}{q \tilde{q}} \leq 
			\Big\lvert
			\frac{\tilde{a}}{\tilde{q}} - \frac{Qa}{q}
			\Big\rvert
			\leq
			\Big\lvert
			Q \theta - \frac{\tilde{a}}{\tilde{q}}
			\Big\rvert
			+
			Q
			\Big\lvert
			\theta - \frac{a}{q}
			\Big\rvert.
		\]
		But $\abs{Q \theta - \tilde{a}/\tilde{q}} \leq N^{-l} (\log N)^{\beta_2}$, and $\abs{\theta - a/q} \leq q^{-2}$, so
		by \eqref{eq:182}
		\[
			\frac{1}{q \tilde{q}} \leq
			N^{-l} (\log N)^{\beta_2} + Q q^{-2}.
		\]
		That is
		\[
                \frac{1}{\tilde{q}} \leq q N^{-l
                } (\log N)^{\beta_2} + Q q^{-1}
			\leq (\log N)^{\beta_2 - \beta} + (\log N)^{\beta' - \beta},
		\]
		since, $(\log N)^\beta \leq q \leq N^{l} (\log N)^{-\beta}$, and $Q \leq (\log N)^{\beta'}$. Thus
		\[
			\frac{1}{\tilde{q}} \leq 2(\log N)^{-\beta_2}.
		\]
		This proves the lemma.
	\end{proof}

\begin{proof}[Proof of Theorem \ref{thm:3}]
Let us denote
  \[
S_N=\sum_{n \in \Omega \cap \ZZ^k} e^{2\pi i P(n)}\varphi(n).
\]
In what follows $C, C', c, c', \ldots$ are constants that appear below. Their value
  are adjusted several times, (but finitely many times) depending
  on the backward induction we use.  The constants $C, C', c, c',
  \ldots$, and the constants implicit  in the $\calO$ notation, may depend on
  $\alpha$, the dimension $k$, the degree $d$ of the
  polynomial $P$, and the function $\varphi$ obeying \eqref{eq:217}, but are independent of $N$ and the coefficients of $P$ and the set $\Omega$.

  In the proof we appeal to a variant of Weyl's inequality with a
  logarithmic loss which can be found in \cite[Section 5]{va} or \cite[the remark after Theorem 1.5]{w0}. 
  \begin{theorem}[Weyl's inequality]
  \label{thm:weyl}
  Let $P$ be a polynomial of degree $d\ge2$ on the real line as in \eqref{eq:12}. Let $R\ge1$ and suppose that there are integers $0\le a\le
q\le R^j$ with $(a, q)=1$, such that
$|\xi_j-a/q|\le q^{-2}$ for some  $2\le j\le d$. Then
\begin{align}
  \label{eq:13}
\Big|\sum_{n = 1}^{R} e^{2\pi i P(n)}\Big|\lesssim R\log R\bigg(\frac 1
q+\frac 1 {R}+\frac{q}{R^{j}}\bigg)^{\varepsilon_d},  
\end{align}
where $\varepsilon_d=\frac{1}{2d^2-2d+1}$. 	
\end{theorem}
 
The proof proceeds in four steps. In the first step we show that inequality \eqref{eq:13} establishes
Theorem \ref{thm:3} for $k=1$ and $|\gamma_0|\ge2$. However, in the second step, we obtain the
less precise result for $k=1$ and  $|\gamma_0|\ge1$ as a consequence of the first step.
In the last two steps we establish Theorem \ref{thm:3} in the full generality.

\noindent {\bf Step 1.} Suppose that $k = 1$, $\varphi\equiv1$ and $|\gamma_0|=j$, where $1\le j\le d$. We
shall prove that for
any $\alpha>0$ there is $\beta_{\alpha}\ge\alpha+1$ such that for every
$\beta\ge\beta_{\alpha}$, if
$(\log N)^{\beta}\le q\le N^{j}(\log N)^{-\beta}$ then
        \begin{align}
          \label{eq:218}
        	\sup_{1\le R\le N}|S_{R}'|\lesssim
		N (\log N)^{-(\alpha+1)},
        \end{align}
 where 
\[
S_R'=		\sum_{n = 1}^{R} e^{2\pi i P(n)}.
\]
Once \eqref{eq:218} is proven we easily obtain \eqref{eq:56} for $k=1$ with
a general function $\varphi$ satisfying
\eqref{eq:217}. Indeed, summing by parts we obtain
\[
S_R=\sum_{n=1}^Re^{2\pi i
  P(n)}\varphi(n)=\sum_{n=1}^R\big(\varphi(n)-\varphi(n+1)\big)S_n'+\mathcal
O\big(|S_R'|\big).
\]
Then inequality \eqref{eq:218} combined with \eqref{eq:217} yields \eqref{eq:56} for $k=1$. The reason why we have taken $\alpha+1$ in \eqref{eq:218} instead of $\alpha$ is purely technical. It was required  to compensate the $\log N$ loss produced by the summation by parts due to \eqref{eq:217}.

It now suffices to establish inequality \eqref{eq:218}. Observe that for $R \le N (\log N)^{-(\alpha+1)}$, the estimate
\eqref{eq:218} holds trivially. Suppose that $N (\log N)^{-(\alpha+1)} \leq R \leq N$. 

For $d=j=1$ and any $\beta\ge\alpha+1$, if $(\log N)^{\beta}\le q\le N(\log N)^{-\beta}$ then we immediately see that
\[
|S_R'|=\Big|\sum_{n = 1}^{R} e^{2\pi i \xi_1n}\Big|\lesssim\|\xi_1\|^{-1}\lesssim q\lesssim N (\log N)^{-(\alpha+1)},
\]
since by \eqref{eq:15} we get
$\|\xi_1\|\ge 1/q-|\xi_1-a/q|\ge 1/q-1/q^2\gtrsim 1/q$, where
$\|\xi\|=\dist(\xi, \ZZ)$.  

For $2\le j\le d$  the desired bound follows by invoking inequality \eqref{eq:13}. Indeed, 
 for any $\beta\ge (\alpha+1)\varepsilon_d^{-1}+d(\alpha +1)$, if $(\log N)^{\beta}\le q\le N^j(\log N)^{-\beta}$,
then we conclude that $q\le R^j$, and \eqref{eq:13} guarantees  
\[
|S_R'
|\lesssim R(\log R)^{-(\beta-j(\alpha+1))\varepsilon_d}\lesssim_{\alpha} N(\log N)^{-(\alpha+1)}.
\]
Finally, we see that
$\beta_{\alpha}=(\alpha+1)(2d^2-d+1)\ge(\alpha+1)$
will work. 
It now remains to prove inequality \eqref{eq:218} for $j=1$ and $d\ge2$. This will be accomplished 
in the next step.

\noindent {\bf Step 2.} We assume that $k = 1$, $\varphi\equiv1$,
$|\gamma_0|=j=1$ and $d\ge2$.  We write
$P(x) = \xi_d x^d + \xi_{d-1} x^{d-1} +\ldots+\xi_{2}x^2+\xi_1x$, and
we assume that
	\[
      	\Big\lvert
		\xi_1 - \frac{a}{q}
		\Big\rvert
		\leq
		\frac{1}{q^2},
	\]
	for some integers $a, q$ such that $0\le a \le q$ and 
	$(a, q) = 1$, with
	\[
		(\log N)^\beta \leq q \leq N (\log N)^{-\beta}.
	\]
	We shall prove
	\begin{equation}
		\label{eq:60}
		\sup_{1\le R\le N}\abs{S_R'} = \calO\big(N (\log N)^{-\alpha}\big), 
	\end{equation}
	for any $\alpha > 0$ and any $\beta \ge \beta_{\alpha}$ as long as $\beta_{\alpha}$ is large enough in terms of $\alpha$. We fix $\beta_1$, to be determined later, and apply
        Dirichlet's principle, obtaining $a_j/q_j$, with $(a_j, q_j) = 1$ and
        $1\le q_j \leq R^j (\log R)^{-\beta_1}$, so that
	\begin{equation}
		\label{eq:152}
		\Big\lvert
		\xi_j - \frac{a_j}{q_j}
		\Big\rvert
		\leq
		\frac{(\log R)^{\beta_1}}{q_jR^{j}}  
	\end{equation}
for all $1<j\le d$.
	There are two cases: 
        \begin{itemize}
        \item \emph{the minor arc} case, when for some
        $1<j\le d$ we have $(\log
        R)^{\beta_1} \leq q_j \leq R^j (\log R)^{-\beta_1}$, 
\item \emph{the major arc} case, when for all 
        $1<j\le d$ we have $1\le q_j \leq (\log R)^{\beta_1}$.
        \end{itemize}
	In the first case we make the choice of $\beta_1$ to be so
        large that the results from the first step can be applied, giving us
	the conclusion \eqref{eq:60} for the given $\alpha > 0$.

	For the case when $1\le q_j \leq (\log R)^{\beta_1}$ for all 
        $1<j\le d$, we will apply Lemma \ref{lem:11}. We
        will take $Q_1=\lcm(q_j: 1<j\le d)$. Write $\theta_j = \xi_j - a_j/q_j$,
        for each $1<j\le d$. Then \eqref{eq:152} implies that
	\begin{equation}
		\label{eq:153}
		\abs{\theta_j}
		\leq
		\frac{(\log R)^{\beta_1}}{q_jR^j}  
	\end{equation}
	and $Q_1 \leq (\log R)^{(d-1)\beta_1}$. We decompose $\ZZ$ modulo $Q_1$, and write $n = Q_1m  + r$,
	with $1 \leq r \leq Q_1$. Thus
	\[
		\sum_{n = 1}^R e^{2\pi i P(n)} 
		=
		\sum_{r = 1}^{Q_1} \sum_{m = 1}^{\lfloor R/Q_1 \rfloor} e^{2\pi i P(Q_1m + r)} + E
	\]
	where $\abs{E} \leq (\log R)^{(d-1)\beta_1}$ because $E$ involves at most $Q_1$ terms. Now,
	\[
		P(Q_1m+r) = \sum_{j=2}^d\xi_j(Q_1m+r)^j + \xi_1
                (Q_1m+r).
	\]
	Hence,
	\begin{align*}
		P(Q_1m + r) &\equiv 
		\sum_{j=2}^d(a_j/q_j) (Q_1m + r)^j + \theta_j (Q_1m+r)^j + \xi_1 (Q_1m+r) 
		\pmod 1\\
		& \equiv\sum_{j=2}^d
		(a_j/q_j) r^j +\sum_{j=2}^d \theta_j (Q_1m +r)^j + \xi_1 Q_1 m +\xi_1r \pmod 1.
\end{align*}
	Thus
	\begin{equation}
		\label{eq:151}
		\sum_{n = 1}^R e^{2\pi i P(n)} 
		=
		\sum_{r = 1}^{Q_1}
		e^{2\pi i \sum_{j=2}^d
		(a_j/q_j) r^j}
		\sum_{m = 1}^{\lfloor R/Q_1 \rfloor}
		A_{m, r} B_{m, r}
		+
		E,
	\end{equation}
	where 
	\begin{align*}
		A_{m, r} = e^{2\pi i \sum_{j=2}^d \theta_j (Q_1m +r)^j}, \qquad
		B_{m, r} = e^{2\pi i (\xi_1 Q_1 m + \xi_1r)}.
	\end{align*}
We estimate the inner sum in \eqref{eq:151}
	by writing it as
	\[
		\sum_{m = 1}^{\lfloor R/Q_1 \rfloor}
		A_{m, r} B_{m, r}
		=
		\sum_{m = 1}^{\lfloor R/Q_1 \rfloor}
		(A_{m,r} - A_{m+1, r}) S_m + \calO\big(|S_{\lfloor R/Q_1 \rfloor}|\big)
	\]
	with $S_m = \sum_{n = 1}^m B_{n, r}$. Since
	\[
		\abs{A_{m, r} - A_{m+1, r}} = \calO\Big(\sum_{j=2}^d\abs{\theta_j} Q_1 R^{j-1} \Big),
	\]
	by \eqref{eq:153}, we obtain
	\[
		\abs{A_{m, r}-A_{m+1, r}} = \calO\big(R^{-1} (\log R)^{d\beta_1}\big).
	\]
	To estimate $S_m$ we are going to apply Lemma
        \ref{lem:11}. Recall that $\beta_1$ has been fixed in the
        minor case, we now set $\beta' = (d-1)\beta_1$ and $Q =
        Q_1$. Since $Q_1 \leq (\log N)^{(d-1)\beta_1}$, we have $Q
        \leq (\log N)^{\beta'}$. Let $\alpha_2 = \alpha +1+
        d\beta_1$ and $\beta_{\alpha_2}$ be determined by
        $\alpha_2$ as in Step 1, and take 
\[
\beta_2\ge\beta_{\alpha_2}.
\]  
Then, for $\beta > 2 \beta_{2}
        + \beta'$ we have $\beta' < \beta$, $\beta_2<\beta-\beta'$, and $2 \beta_2 < \beta$, thus by Lemma \ref{lem:11}, 
	we obtain
	\[
          \Big|Q \xi_1 -\frac{\tilde{a}}{\tilde{q}}\Big|\le \frac{1}{\tilde{q}^2},
	\]
	for some integers $0\le \tilde{a}\le \tilde{q}$ such that $(\tilde{a}, \tilde{q})=1$ and $(\log N)^{\beta_2}\le \tilde{q}\le N(\log N)^{-\beta_2}$. 
	Hence, by Step 1 applied to the polynomial $x \mapsto \xi_1Q x $, we get
	\[
		\abs{S_m} = \calO\big(N (\log N)^{-\alpha_2}\big).
	\]
	Therefore,
	\[
		\Big\lvert
		\sum_{n=1}^R e^{2\pi i P(n)}
		\Big\rvert
		=
		\calO\big(Q_1 R^{-1} (\log R)^{d\beta_1} R Q_1^{-1} N(\log N)^{-\alpha_2}\big),
	\]
	since 
	\[
		\sum_{m = 1}^{\lfloor R/Q_1 \rfloor} N (\log N)^{-\alpha_2} 
		= 
		\calO\big(R Q_1^{-1} N(\log N)^{-\alpha_2}\big).
	\]
	Hence,
	\[
		\Big\lvert
		\sum_{n=1}^R e^{2\pi i P(n)}
		\Big\rvert
		=
		\calO\big(N (\log N)^{d\beta_1 - \alpha_2}\big),
	\]
	which gives us the desired conclusion and completes the proof of Theorem \ref{thm:3} for $k=1$. 
	In the next two steps we will handle  the general case. 
 
	\noindent {\bf Step 3.} Let  $k > 1$ and suppose that  there is a multi-index $\gamma_0$ such that $\abs{\gamma_0} = d$ and
	\[
		\Big\lvert
		\xi_{\gamma_0} - \frac{a}{q}
		\Big\rvert
		\leq
		\frac{1}{q^2},
	\]
	for some integers  $0\le a\le q$ with $(a,q)=1$ and $(\log N)^{\beta} \leq q
        \leq N^d (\log N)^{-\beta}$. We will need the following lemma.  
	\begin{lemma}[{\cite[Lemma 1]{SW0}}]
		\label{lem:16}
		For each $\gamma_0$ such that $\abs{\gamma_0} = d$,
                there exist $\nu$ linear transformations $L_1, \ldots,
                L_\nu$ of $\RR^k$ that have integer coefficients and
                determinant 1 (so that each $L_j$ is an automorphism
                of $\ZZ^k$) and integers $c_0, \ldots, c_\nu$, with
                $c_0 \neq 0$, so that if $\theta$ is the coefficient
                of $x^{\gamma_0}$ of $P(x)$, and $\sigma_j$ is the
                coefficient of $x_1^d$ of $P(L_j x)$, then
		\[
			c_0 \theta = c_1 \sigma_1 + \ldots + c_\nu\sigma_\nu.
		\]
		The operators $L_1, \ldots, L_\nu$, and integers $c_0,
                \ldots, c_\nu$, depend only on $k$, $d$, and
                $\gamma_0$. Moreover, $\nu$ is the dimension of the
                vector space of polynomials in $\RR^k$ which are
                homogeneous of degree $d$.
	\end{lemma}
We shall apply Lemma \ref{lem:16} with $\theta=\xi_{\gamma_0}$.
	Now, for $\beta_1$ sufficiently large, determined below, apply
        Dirichlet's principle to each $\sigma_j$ to get $a_j/q_j$ so that
	$(a_j, q_j) = 1$ and
	\begin{equation}
		\label{eq:183}
		\Big\lvert
		\sigma_j
		-
		\frac{a_j}{q_j}
		\Big\rvert
		\leq
		\frac{(\log N)^{\beta_1}}{q_jN^d}, 
	\end{equation}
	and $1\le q_j \leq N^d (\log N)^{-\beta_1}$, for $j = 1,
        \ldots, \nu$. There are two cases. The first case, the
        minor case, that is when
	$q_j \geq (\log N)^{\beta_1}$ for at least one $j \in \{1, \ldots, \nu\}$, we write $L = L_j$ and 
	 $\tilde{P}(x) = P(L(x))$, $\tilde{\varphi}(x) = \varphi(L(x))$ and $\tilde{\Omega} = L^{-1}[\Omega]$. We observe that
	\begin{equation}
		\label{eq:157}
		\sum_{n \in \Omega \cap \ZZ^k} e^{2\pi i P(n)} \varphi(n)
		= 
		\sum_{n \in \tilde{\Omega} \cap \ZZ^k} e^{2\pi i \tilde{P}(n)}\tilde{\varphi}(n),
	\end{equation}
	and since $\Omega \subseteq \big\{x \in \RR^k : \norm{x} \leq N \big\}$, we have
	$\tilde{\Omega} \subseteq \big\{x \in \RR^k : \norm{x} \leq \tilde{c} N\big\}$ for some $\tilde{c} > 0$.

	Next, for each $n \in \ZZ^k$, write $n = (n_1, n')$, where $n_1 \in \ZZ$ and $n' \in \ZZ^{k-1}$, getting
	\begin{equation}
		\label{eq:158}
		\sum_{n \in \tilde{\Omega} \cap \ZZ^k}
		e^{2\pi i \tilde{P}(n)}\tilde{\varphi}(n)
		=
		\sum_{\atop{n' \in \ZZ^{k-1}}{|n'|\le \tilde cN}}
		\sum_{\atop{n_1 \in \ZZ}{(n_1, n') \in \tilde{\Omega}}}
		e^{2\pi i \tilde{P}(n_1, n')}\tilde{\varphi}(n_1, n').
	\end{equation}
	Thus it suffices to have
	\begin{equation}
		\label{eq:159}
		\Big\lvert
\sum_{\atop{n_1 \in \ZZ}{(n_1, n') \in \tilde{\Omega}}}
		e^{2\pi i \tilde{P}(n_1, n')}\tilde{\varphi}(n_1, n')
		\Big\rvert
		=
		\calO\big(N (\log N)^{-\alpha} \big),
	\end{equation}
	where the implicit constant is independent of $n'$. This will
        give the desired result for \eqref{eq:157}, since there are at
        most $\calO\big(N^{k-1}\big)$ terms that appear in the first
        summation in \eqref{eq:158} as
        $\norm{n'} \leq \tilde{c} N$. We now apply the results from
        Step 1 to the one-dimensional polynomial
        $n_1 \mapsto \tilde{P}(n_1, n')$.  Observe that
        $\tilde{P}(n_1, n') = \sigma_j n_1^d + \text{lower order terms
          of $n_1$}$.  We choose any $\beta_1 \ge\beta_{\alpha}$, where $\beta_{\alpha}$ is as in
        Step 1 and we obtain \eqref{eq:159}, which leads to the desired
        conclusion.

	The second case, the major case, where $1\le q_j
        \leq (\log N)^{\beta_1}$ for all $1\le j\le \nu$, will be seen to be empty,
	where $\beta_1$ has been fixed and $\beta$ is chosen large enough. In fact, by Lemma \ref{lem:16}
	\begin{equation}	
		\label{eq:160}
		\theta = c_0^{-1} \big(c_1 \sigma_1 + \ldots + c_\nu \sigma_\nu\big),
	\end{equation}
	so if we write $a'/q' = c_0^{-1}\big(c_1 a_1/q_1 + \ldots +c_\nu a_\nu /q_{\nu}\big)$, with $(a', q') = 1$,
	then $q' \leq c_0 (\log N)^{\beta'}$, with $\beta' = \nu
        \beta_1$ since $q_j \leq (\log N)^{\beta_1}$ for all $1\le
        j\le \nu$.

	Now there are two subcases: $a'/q' = a/q$ and $a'/q' \neq a/q$. The first is not possible, except for finitely many
	$N$, since 
	\[
		(\log N)^\beta \leq q = q' \leq c_0 (\log N)^{\beta'}
	\]
	and we may take $\beta > \beta'$. In the second case, taking
	into account  \eqref{eq:160} and \eqref{eq:183}, we write
	\[
		\frac{1}{q q'} 
		\leq
		\Big\lvert
		\frac{a'}{q'} - \frac{a}{q}
		\Big\rvert
		\leq
		\Big\lvert
		\theta - \frac{a}{q}
		\Big\rvert
		+
		\Big\lvert
		\theta - \frac{a'}{q'}
		\Big\rvert
		\leq
		\frac{1}{q^2} + C N^{-d} (\log N)^{\beta_1}.
                \]
                 Thus
	\[
		c_0^{-1} (\log N)^{-\beta'} \leq 
		\frac{1}{q'} \leq \frac{1}{q} + C q N^{-d} (\log N)^{\beta_1}
		\leq C' \big((\log N)^{-\beta}  + (\log N)^{\beta_1 - \beta}\big)
	\]
	since $(\log N)^\beta \leq q \leq N^d (\log N)^{-\beta}$. Thus, it is enough to take $\beta > (\nu + 1)\beta_1$
	to make this impossible, since $\beta > \beta' = \nu \beta_1$  and $\beta > \beta' + \beta_1$. This concludes Step 3.

	\noindent {\bf Step 4.} We assume that $k > 1$ and
        $\abs{\gamma_0} < d$. This step combines ideas of both Step 2 and
        Step 3. The argument is by backward induction on the degree of
        $\abs{\gamma_0}$. The result from
        Step 3 establishes the first step of the backward
        induction for $|\gamma_0|=d$. Assume that Theorem \ref{thm:3} holds for
        all $\abs{\gamma_0}=j$ such that $l<j\le d$.  Our aim now is to deduce it holds for
        $\abs{\gamma_0} = l$.

	Write $P = P_0 + P_1$, with 
\[
P_0(x) = \sum_{l<\abs{\gamma} \le d} \xi_\gamma x^\gamma\quad \text{and}
\quad
P_1(x) = \sum_{\abs{\gamma} \leq l} \xi_\gamma x^\gamma.
\]
 On $P_1$ we apply Lemma \ref{lem:16} and argue as in Step 3.
	This gives us an automorphism $L$ of $\ZZ^k$, so that if $\tilde{P}(x) = P(Lx)$, $\tilde{P}_0(x) = P_0(L x)$,
	$\tilde{P}_1(x) = P_1(Lx)$, then $\tilde{P}_1(x) = \theta x_1^{l} + R(x)$, where $R$ is a polynomial of degree
	$\le l$, but of degree $\le l-1$ in $x_1$. Also
	\[
		\Big\lvert
		\theta - \frac{a_1}{q_1}
		\Big\rvert
		\leq
		\frac{(\log N)^{\beta_0}}{q_1N^l}  
	\]
	with integers $0\le a_1\le q_1$ so that $(a_1, q_1) = 1$, and $(\log N)^{\beta_0} \leq q_1 \leq N^{l} (\log N)^{-\beta_0}$ as long as
	\begin{equation}
		\label{eq:161}
		(\nu_1 + 1) \beta_0 \leq \beta,
	\end{equation}
	where $\nu_1$ is the dimension of the space of homogeneous
        polynomials of degree $l$ in $\RR^k$. Let us emphasize that the parameter $\beta_0$
        in \eqref{eq:161} plays the role of $\beta_1$ from the
        previous step.  

        Next, we choose $\beta_1 > 0$ whose value will be determined
        later, and apply Dirichlet's principle to all the coefficients
        $\tilde{\xi}_\gamma$ of $\tilde{P}_0$ with $l<\abs{\gamma} \le
        d$. Thus we can find $a_\gamma/q_\gamma$, so that
	\[
		\Big\lvert
		\tilde{\xi}_\gamma - \frac{a_\gamma}{q_\gamma}
		\Big\rvert
		\leq
		\frac{(\log N)^{\beta_1}}{q_\gamma N^{|\gamma|}}  
	\]
	with $(a_\gamma, q_\gamma) = 1$ and $1\le q_\gamma \leq
        N^{|\gamma|} (\log N)^{-\beta_1}$. We also set $\theta_\gamma
        = \tilde{\xi}_\gamma - a_\gamma/q_\gamma$. If $(\log
        N)^{\beta_1} \leq q_\gamma$, for some $\gamma$ such that
        $l<|\gamma|\le d$, then by the induction hypothesis, we obtain
        \eqref{eq:56} provided that $\beta_1$ is sufficiently large,
        and we are done. We fix this $\beta_1$.

Let us now define $\beta' = \beta_1 \nu l$, where
$\nu=|\{\gamma\in\NN_0^k: l<|\gamma|\le d\}|$, and $\alpha_2=(k+1)(\nu+1)
\beta_1  + \alpha$ and $\beta_{\alpha_2}$  be determined 
by $\alpha_2$ as in Step 1, and take
\begin{align}
  \label{eq:16}
 \beta_2\ge\beta_{\alpha_2}.
\end{align}
Suppose that $1\le q_\gamma \leq (\log N)^{\beta_1}$, for all
$\gamma$ such that $l<\abs{\gamma} \le d$. Let $q' = \lcm (q_\gamma :
l<\abs{\gamma} \le d)$. Then $q' \leq (\log N)^{\nu \beta_1}$. Set $Q = (q')^{l}$,
then $Q \leq (\log N)^{\beta'}$. Next, we choose 
\begin{align}
\label{eq:14}
  \beta_0>2\beta_2+\beta'.
\end{align}
Then $\beta_0>\beta'$, $\beta_2<\beta_0-\beta'$ and $2\beta_2<\beta_0$, thus we can 
we apply Lemma \ref{lem:11} to get $\tilde{a}/\tilde{q}$ satisfying
	\[
		\Big\lvert
		Q \theta - \frac{\tilde{a}}{\tilde{q}}
		\Big\rvert
		\leq
		\frac{(\log N)^{\beta_2}}{\tilde{q}N^l} 
		\leq
		\frac{1}{\tilde{q}^2}
	\]
	with $(\tilde{a}, \tilde{q}) = 1$, $(\log N)^{\beta_2} \leq \tilde{q} \leq N^{l}
        (\log N)^{-\beta_2}$. Next, we break the sum
	\[
		\sum_{n \in \tilde{\Omega} \cap \ZZ^k} e^{2\pi i \tilde{P}(n)}\tilde{\varphi}(n)
	\]
	into essentially a sum of disjoint boxes of side-length $q'$. For $n \in \ZZ^k$, we write $n = q' m + r$,
	with $r \in \NN^k_{q'}$. Let
	\[
		\tilde{\Omega}_{q'} = \big\{m \in \ZZ^k : q' m + \ZZ^k_{q'} \subseteq \tilde{\Omega}\big\}.
	\]
	Then
	\[
		\sum_{n \in \tilde{\Omega} \cap \ZZ^k} e^{2\pi i \tilde{P}(n)} \tilde{\varphi}(n)
		=
		\sum_{r \in \ZZ^k_{q'}}
		\sum_{m \in \tilde{\Omega}_{ q'}}
		e^{2\pi i \tilde{P}(q'm + r)}\tilde{\varphi}(q'm+r)
		+
		\sum_{n \in \Delta} e^{2\pi i \tilde{P}(n)}\tilde{\varphi}(n).
	\]
	The residual set of points $\Delta$ are lattice points in
        $\tilde{\Omega}$ whose distance from the boundary of
        $\tilde{\Omega}$ is $\calO(q')=\calO(N^{1-3\sigma})$ for any
        $0<\sigma <1/3$. Hence, by Proposition \ref{prop:201} there
        are $\calO\big(N^{k-\sigma} \big)$ such points. Thus that sum
        contributes $\calO\big(N^{k-\sigma} \big)$, which is
        $\calO\big(N^k (\log N)^{-\alpha}\big)$ for every $\alpha >
        0$.

	Therefore, we are reduced to considering
	\begin{equation}
		\label{eq:163}
		\sum_{r \in \ZZ^k_{q'}} \sum_{m \in \tilde{\Omega}_{ q'}}
		e^{2\pi i \tilde{P}(q'm + r)}\tilde{\varphi}(q'm+r).
	\end{equation}
	Let us fix $r \in \ZZ^k_{q'}$ and $m' \in \ZZ^{k-1}$, and write
	\[
		\sum_{\atop{m_1 \in \ZZ}{(m_1,m') \in \tilde{\Omega}_{ q'}}}
		e^{2\pi i \tilde{P}(q'm + r)}\tilde{\varphi}(q'm+r)
		=
		\sum_{m_1 = M_0}^{M_1}
		A_{m_1} B_{m_1},
	\]
	where $\{M_0, \ldots, M_1\} = \big\{m_1\in\ZZ : (m_1, m') \in \tilde{\Omega}_{q'}\big\}$ and
	\[
		A_{m_1} =  e^{2\pi i \sum_{l<\abs{\gamma} \le d}\theta_\gamma ( q'm + r)^\gamma}, \quad \text{and} \quad
		B_{m_1} = e^{2\pi i Q \theta m_1^{l} + R(m_1)}\tilde{\varphi}(q'm+r),
	\]
	and $R$ is a polynomial in $m_1$ of degree $\le l-1$,
        depending on $r$ and $m'$. Now, by summation by parts we get
\[
		\sum_{m_1 = M_0}^{M_1}
		A_{m_1} B_{m_1}=		\sum_{m_1 = M_0}^{M_1}
		(A_{m_1} - A_{m_1+1}) S_{m_1} + \calO\big(|S_{M_1}|\big)+\calO\big(|S_{M_0-1}|\big)
	\]
	with $S_m = \sum_{n = 1}^m B_{n}$. But
	\begin{align*}
		\abs{A_{m_1} - A_{m_1+1}} &= \calO\Big(\sum_{l<\abs{\gamma} \le d}q'|\theta_{\gamma}| N^{|\gamma|-1}  \Big)\\
		&= \calO\big((\log N)^{\nu \beta_1} N^{|\gamma|-1} N^{-|\gamma|} (\log N)^{\beta_1} \big)\\
		&= \calO\big(N^{-1} (\log N)^{(\nu+1) \beta_1 }\big).
	\end{align*}
	However,
	\[
		S_{m_1} = \Big\lvert \sum_{n=1}^{m_1} B_{n} \Big \rvert = \calO\big(m_1 (\log m_1)^{-\alpha_2}\big),
	\]
	by the one-dimensional result applied to $Q \theta$, with
        $\alpha_2=(k+1)(\nu+1) \beta_1+\alpha$. Therefore,
	\[
		\Big\lvert 
		\sum_{m_1 = M_0}^{M_1} A_{m_1} B_{m_1} 
		\Big\rvert
		=
		\calO\big(N (\log N)^{(\nu+1) \beta_1  - \alpha_2}\big),
	\]
	and summing up with respect to $m'$ and $r$, we get that \eqref{eq:163} is
	\[
		\calO\big(N^k (\log N)^{(k+1)(\nu+1) \beta_1 - \alpha_2}\big)=\calO\big(N^k (\log N)^{- \alpha}\big).
	\]
	Thus taking into account \eqref{eq:161}, \eqref{eq:16} and
\eqref{eq:14}, we only need  to make $\beta$ large enough, to make
$\beta_0$ large enough, then to make $\beta_2$ large enough to get our
desired conclusion.
\end{proof}

\section{Vector-valued maximal estimates for averaging operators}
\label{sec:4}
This section is intended to prove Theorem \ref{thm:10}. We begin by setting up notation and terminology.
For any $x\in\ZZ^d$ and any function $f: \ZZ^d \rightarrow \CC$ with a finite support we have
$$
M_N f(x) = K_N * f(x),
$$
where $K_N$ is a kernel defined by
\begin{align}
	\label{eq:82}
	K_N(x) = N^{-k} \sum_{y \in \NN_N^k} \delta_{\calQ(y)},
\end{align}
where $\delta_y$ denotes the Dirac delta at $y \in \ZZ^k$ and $\calQ$ is the canonical polynomial defined
in Section \ref{sec:5}. Let $m_N$ denote the discrete Fourier transform of $K_N$, i.e.,
$$
m_N(\xi) = N^{-k} \sum_{y \in \NN_N^k} e^{2\pi i \sprod{\xi}{\calQ(y)}}, \quad\text{for}\quad\xi\in\TT^d.
$$
Finally, we define
$$
\Phi_N(\xi) = \int_{[0, 1]^k} e^{2\pi i \sprod{\xi}{\calQ(N y)}} {\: \rm d}y, \quad\text{for}\quad\xi\in\RR^d.
$$
Using a multi-dimensional version of van der Corput lemma (see \cite{ccw, bigs, sw}) we may estimate
\begin{equation}
\label{eq:83}
	\abs{\Phi_N(\xi)}
	\lesssim
	\min\big\{1, \norm{N^A \xi}_{\infty}^{-1/d} \big\},
\end{equation}
where $A$ is the $d\times d$ diagonal matrix defined in \eqref{eq:25}. Additionally, we have
\begin{equation}
	\label{eq:84}
	\abs{\Phi_N(\xi) - 1}
	\lesssim
	\min\big\{1, \norm{N^A \xi}_{\infty}\big\}.
\end{equation}
For $q\in\NN$ let us recall that
\[
	A_q=\big\{a\in\NN_q^d: \gcd\big(q, (a_{\gamma}: \gamma\in\Gamma)\big)=1\big\}.
\]
Next, for $q \in \NN$ and $a \in A_q$ we define the \emph{Gaussian sum} 
$$
	G(a/q) = q^{-k} \sum_{y \in \NN^k_q} e^{2\pi i \sprod{(a/q)}{\calQ(y)}}.
$$
By  multi-dimensional 
Weyl's inequality (see \cite[Proposition 3]{SW0}), there exists $\delta>0$ such that
\begin{equation}
	\label{eq:20}
	\lvert G(a/q) \rvert \lesssim q^{-\delta}.
\end{equation}

We shall prove that for every $p\in(1, \infty)$ there is
a constant $C_p>0$ such that for every sequence $\big(f_t:
t\in\NN\big)\in\ell^{p}\big(\ell^2\big(\ZZ^d\big)\big)$ of
non-negative functions with finite supports we have
\begin{align}
  \label{eq:6}
	\Big\|\Big(\sum_{t\in\NN}\sup_{n\in\NN_0}|M_{2^n}f_t|^2\Big)^{1/2}\Big\|_{\ell^p}
	\le
	C_p
	\Big\|\big(\sum_{t\in\NN}|f_t|^2\big)^{1/2}\Big\|_{\ell^p}.  
\end{align}
We begin by proving the following proposition.
\begin{proposition}
  \label{pr:0}
 There is a constant $C>0$ such
  that for every $N\in\NN$ and for every $\xi\in [-1/2, 1/2)^d$ satisfying 
        $$
	\Big\lvert \xi_\gamma - \frac{a_\gamma}{q} \Big\rvert \leq
        L_1^{-|\gamma|}L_2
	$$
	for all $\gamma \in \Gamma$, where  $1\le q\le L_3\le N^{1/2}$, $a\in
        A_q$, $L_1\ge N$  and $L_2\ge1$ we have
	\[
          \big|m_N(\xi)-G(a/q)\Phi_{N}(\xi-a/q)\big|\le
          C\Big(L_3N^{-1}+L_2L_3N^{-1}\sum_{\gamma \in
            \Gamma}\big(N/L_1\big)^{|\gamma|}\Big).
	\]
\end{proposition}
\begin{proof}
	Let $\theta = \xi - a/q$, then
	$$
	N^{-k}\sum_{y \in \NN_N^k} e^{2\pi i \sprod{\xi}{\calQ(y)}} 
	=
	q^{-k}\sum_{r \in \NN_q^k}
	e^{2\pi i \sprod{(a/q)}{\calQ(r)}}
	\cdot \Big(q^kN^{-k}\sum_{\atop{y \in \NN_N^k}{qy+r\in[1, N]^k}}
	e^{2\pi i \sprod{\theta}{\calQ(qy+r)}}\Big).
	$$
	If $ \norm{q y + r}, \norm{qy} \leq N$ then
        \begin{align*}
          \big\lvert
	\sprod{\theta}{\calQ(q y + r)} - \sprod{\theta}{\calQ(q y)}
	\big\rvert
	&\lesssim
	\norm{r}
	\sum_{\gamma \in \Gamma}
	\abs{\theta_\gamma}
	\cdot
	N^{(\abs{\gamma} - 1)}\\
	&\lesssim
	q \sum_{\gamma \in \Gamma}
	L_1^{-\abs{\gamma}}L_2 N^{(\abs{\gamma}-1)}\\
	&\lesssim
	L_2L_3N^{-1}\sum_{\gamma \in \Gamma}\big(N/L_1\big)^{\abs{\gamma}}.
        \end{align*}
	Thus
    \begin{align*}
        N^{-k}\sum_{y \in \NN_N^k} e^{2\pi i \sprod{\xi}{\calQ(y)}} 
		&=G(a/q)\cdot q^kN^{-k}\sum_{\atop{y \in \NN_N^k}{qy\in[1, N]^k}}
          e^{2\pi i \sprod{\theta}{\calQ(qy)}}\\
        &\phantom{=}+\mathcal O\Big(qN^{-1}+L_2L_3N^{-1}\sum_{\gamma \in \Gamma}\big(N/L_1\big)^{|\gamma|}\Big).  
	\end{align*}
By the mean value theorem one  replaces the sum on the right-hand side by the integral. Indeed, we have
  \begin{align*}
    & \Big|\sum_{y\in(0, \lfloor N/q\rfloor]^k}
    e^{2\pi i \sprod{\theta}{\calQ(qy)}}-\int_{[0, N/q]^k}e^{2\pi i \theta\cdot\calQ(qt)} {\: \rm d}t\Big|\\
    & \qquad \qquad = \Big|\sum_{y\in(0, \lfloor N/q\rfloor]^k} e^{2\pi i
      \sprod{\theta}{\calQ(qy)}}-\sum_{y\in(0, \lfloor
      N/q\rfloor-1]^k}\int_{y+(0, 1]^k}e^{2\pi i
      \theta\cdot\calQ(qt)}{\: \rm d} t\Big|+\mathcal O\big((N/q)^{k-1}\big)\\
	& \qquad \qquad =\Big|\sum_{y\in(0, \lfloor N/q\rfloor-1]^k} \int_{(0, 1]^k}e^{2\pi i
      \theta\cdot\calQ(qy)}-e^{2\pi i
      \theta\cdot\calQ(q(t+y))}{\: \rm d} t\Big|+\mathcal
    O\big((N/q)^{k-1}\big)\\
	& \qquad \qquad 
=\mathcal O\Big((N/q)^{k-1}+(N/q)^kL_2L_3N^{-1}\sum_{\gamma \in \Gamma}\big(N/L_1\big)^{|\gamma|}\Big).
  \end{align*}
	This completes the proof.
\end{proof}
Fix $p\in(1, \infty)$ and let $l\ge10$ be a large integer adjusted to
$p$, whose precise value will be specified later.  Let $\rho>0$ be the
parameter as in Theorem \ref{th:3} and suppose that
\begin{align}
  \label{eq:211}
  10\rho l=1.
\end{align}
For $\chi\in(0, 1/10)$ and for every $n\in\NN_0$ we introduce the multiplier
\begin{align}
  \label{eq:9}
  \Xi_{n}(\xi)=\sum_{a/q\in\mathscr U_{n^l}}
	\eta_n(\xi - a/q)
\end{align}
with $\mathscr U_{n^l}$ defined in \eqref{eq:156} and
\[
	\eta_n(\xi) = \eta\big(2^{n(A-\chi I)} \xi\big).
\]
Due to \eqref{eq:2} and \eqref{eq:211} we have
\begin{align}
  \label{eq:7}
  |\mathscr U_{n^l}|\lesssim e^{(d+1)n^{\rho l}}\lesssim e^{(d+1)n^{1/10}}.
\end{align}
Moreover, Theorem \ref{th:3} yields the following estimate
\begin{align}
\label{eq:35}
  \big\|\mathcal
  F^{-1}\big(\Xi_{n}\hat{f}\big)\big\|_{\ell^p}\lesssim_{\rho, p} \log(n+2)\|f\|_{\ell^p},
\end{align}
because for every $\gamma\in\Gamma$ and sufficiently large $n\in\NN_0$ we have
$2^{-n(|\gamma|-\chi)}\le e^{-n^{2\rho l}}\le e^{-n^{1/5}}$.

  Observe that the left-hand side of \eqref{eq:6} can be dominated as follows
	\begin{equation}
	\label{eq:36}
	\begin{aligned}
	\bigg\|\Big(\sum_{t\in\NN}\sup_{n\in\NN_0}|M_{2^n}f_t|^2\Big)^{1/2}\bigg\|_{\ell^p}
	&\le
        \bigg\|\Big(\sum_{t\in\NN}\sup_{n\in\NN_0}\big|\mathcal
	F^{-1}\big(m_{2^n}(1-\Xi_{n})\hat{f_t}\big)\big|^2\Big)^{1/2}\bigg\|_{\ell^p}\\
	&\phantom{\le}+\bigg\|\Big(\sum_{t\in\NN}\sup_{n\in\NN_0}\big|\mathcal
 	F^{-1}\big(m_{2^n}\Xi_{n}\hat{f_t}\big)\big|^2\Big)^{1/2}\bigg\|_{\ell^p}.
\end{aligned}
\end{equation}
Using the terminology from the circle method of the Hardy and Littlewood the first term in
\eqref{eq:36} corresponds to the minor arcs and the second term
corresponds to the major arcs.
\subsection{The estimate for the first term in \eqref{eq:36}}
Since
\[
\bigg\|\Big(\sum_{t\in\NN}\sup_{n\in\NN_0}\big|\mathcal
  F^{-1}\big(m_{2^n}(1-\Xi_{n})\hat{f_t}\big)\big|^2\Big)^{1/2}\bigg\|_{\ell^p}\le
\sum_{n\in\NN_0}\bigg\|\Big(\sum_{t\in\NN}\big|\mathcal
  F^{-1}\big(m_{2^n}(1-\Xi_{n})\hat{f_t}\big)\big|^2\Big)^{1/2}\bigg\|_{\ell^p},
\]
it suffices to show that
\begin{align}
  \label{eq:10}
  \bigg\|\Big(\sum_{t\in\NN}\big|\mathcal
  F^{-1}\big(m_{2^n}(1-\Xi_{n})\hat{f_t}\big)\big|^2\Big)^{1/2}\bigg\|_{\ell^p}\lesssim
  (n+1)^{-2}
  \Big\|\big(\sum_{t\in\NN}|f_t|^2\big)^{1/2}\Big\|_{\ell^p},  
\end{align}
which in view of \eqref{eq:69} holds, if we can prove that
\begin{align}
  \label{eq:61}
  \big\|\mathcal
  F^{-1}\big(m_{2^n}(1-\Xi_{n})\hat{f}\big)\big\|_{\ell^p}\lesssim (n+1)^{-2}\|f\|_{\ell^p}.
\end{align}
Indeed, by \eqref{eq:35} we have, for every $p\in(1, \infty)$, that
	\begin{align}
	\nonumber
	\big\|\mathcal F^{-1}\big(m_{2^n}(1-\Xi_{n})\hat{f}\big)\big\|_{\ell^p} 
	&\le
	\big\|
  	M_{2^n}f\big\|_{\ell^p}+\big\| M_{2^n}\big(
  	\mathcal F^{-1}\big(\Xi_{n}\hat{f}\big)\big)\big\|_{\ell^p}\\
	\label{eq:62}
	&\lesssim 
	\log(n+2)\|f\|_{\ell^p}.
\end{align}
We show that it is possible to improve  estimate
\eqref{eq:62} 
for $p=2$. Namely, by Theorem \ref{thm:3} for every large $\alpha>0$, which will be specified
later, and for all $n\in\NN_0$  we have
\begin{align}
  \label{eq:63}
\sup_{\xi\in\TT^d} \big|
 m_{2^n}(\xi)(1-\Xi_{n}(\xi))\big|\lesssim (n+1)^{-\alpha}.
\end{align}
In order to do so,
by Dirichlet's principle for every $\gamma\in\Gamma$ and $\beta>0$, we have
\[
|\xi_{\gamma}-a_{\gamma}/q_{\gamma}|\le q_{\gamma}^{-1}n^{\beta}2^{-n|\gamma|},
\] 
where $1\le q_{\gamma}\le n^{-\beta}2^{n|\gamma|}$. To apply Theorem \ref{thm:3}
we must show that there exists 
$\gamma\in\Gamma$ such that $n^{\beta}\le q_{\gamma }\le
n^{-\beta}2^{n|\gamma|}$. Suppose for a contradiction that for every
$\gamma \in \Gamma$  we have $1\le q_{\gamma }<n^{\beta} $. Then for
some 
$q\le \lcm(q_{\gamma}: \gamma\in\Gamma)\le n^{\beta d}$ we have 
\[
|\xi_{\gamma}-a_{\gamma}'/q|\le n^{\beta}2^{-n|\gamma|},
\] 
where $\gcd\big(q, \gcd({a_{\gamma}'}:
\gamma\in\Gamma)\big)=1$.
Hence, taking $a'=(a_{\gamma}': \gamma\in\Gamma)$ we have
$a'/q\in\mathscr U_{n^l}$ provided that $\beta d<l$. On the
other hand, if $1-\Xi_{n}(\xi)\not=0$ then there exists $\gamma\in\Gamma$ such that
\[
|\xi_{\gamma}-a_{\gamma}'/q|> 2^{-n(|\gamma|-\chi)}/16d.
\] 
Therefore,
\[
	2^{\chi n}<16dn^{\beta},
\] 
 but this gives a contradiction for sufficiently large $n\in\NN$.
Hence, there exists $\gamma\in\Gamma$ such that $n^{\beta}\le q_{\gamma }\le n^{-\beta}2^{n|\gamma|}$,
and Theorem \ref{thm:3} yields \eqref{eq:63}. Consequently, by Plancherel's theorem we obtain
\begin{align}
\label{eq:64}
\big\|\mathcal
  F^{-1}\big(m_{2^n}(1-\Xi_{n})\hat{f}\big)\big\|_{\ell^2}\lesssim(1+n)^{-\alpha} \|f\|_{\ell^2}.
\end{align}
Interpolating \eqref{eq:64} with \eqref{eq:62} we get 
\[
 \big\|\mathcal F^{-1}\big(m_{2^n}(1-\Xi_{n})\hat{f}\big)\big\|_{\ell^p}\lesssim(1+n)^{-c_p\alpha} \|f\|_{\ell^p},
\]
for some $c_p>0$. Choosing $\alpha>0$ and $l\in\NN$ appropriately
large and adjusted to $p$, we get \eqref{eq:61}.
 
\subsection{The estimate for the second term in \eqref{eq:36}}
Note that for any $\xi\in\TT^d$ satisfying
\[
|\xi_{\gamma}-a_{\gamma}/q|\le (8d)^{-1}2^{-n(|\gamma|-\chi)}
\] 
for every $\gamma\in \Gamma$, with $1\le q\le e^{n^{1/10}}$, we have 
\begin{align}
\label{eq:37}
  m_{2^n}(\xi)=G(a/q)\Phi_{2^n}(\xi-a/q)+E_{2^n}(\xi),
\end{align}
where 
\begin{align}
\label{eq:38}
  |E_{2^n}(\xi)|\lesssim 2^{-n/2}.
\end{align}
These two properties \eqref{eq:37} and \eqref{eq:38} follow from
Proposition \ref{pr:0} with $L_1=2^n$, $L_2=2^{\chi n}$ and
$L_3=e^{n^{1/10}}$, because
\[
|E_{2^n}(\xi)|\lesssim L_2L_32^{-n}\lesssim
\big(e^{-n((1-\chi)\log 2-2n^{-9/10})}\big)\lesssim 2^{-n/2},
\]
which holds for sufficiently large $n\in\NN$ and $\chi>0$ sufficiently small. 

For each $n \in \NN_0$, let us introduce the multiplier
\[
	\nu_{2^n}(\xi)=\sum_{a/q\in\mathscr U_{n^l}}G(a/q)\Phi_{2^n}(\xi-a/q) \eta_n(\xi - a/q).
\]
In view of \eqref{eq:37}, we have
\begin{align*}
 |m_{2^n}(\xi)\Xi_n(\xi)-\nu_{2^n}(\xi)|\lesssim 2^{-n/2},
\end{align*}
consequently by Plancherel's theorem
\begin{align}
  \label{eq:70}
  \big\|\mathcal
  F^{-1}\big((m_{2^n}\Xi_n-\nu_{2^n})\hat{f}\big)\big\|_{\ell^2}\lesssim
2^{-n/2}\|f\|_{\ell^2}.
\end{align}
Moreover, by Theorem \ref{th:3} we have
\begin{align*}
  \big\|\mathcal
  F^{-1}(m_{2^n}\Xi_n\hat{f})\big\|_{\ell^p}\lesssim  \log(n+2)\|f\|_{\ell^p},
\end{align*}
which together with \eqref{eq:7} gives
\begin{align*}
\big\|\mathcal F^{-1}(\nu_{2^n}\hat{f})\big\|_{\ell^p}
\lesssim
  |\mathscr U_{n^l}|\cdot\|f\|_{\ell^p}\lesssim e^{(d+1)n^{1/10}}\|f\|_{\ell^p}.
\end{align*}
Hence,
\begin{align}
  \label{eq:72}
  \big\|\mathcal
  F^{-1}\big((m_{2^n}\Xi_n-\nu_{2^n})\hat{f}\big)\big\|_{\ell^p}\lesssim
e^{(d+1)n^{1/10}}\|f\|_{\ell^p}.
\end{align}
Interpolating now \eqref{eq:70} with \eqref{eq:72} we  conclude that for some $c_p>0$,
\begin{align}
  \label{eq:73}
  \big\|\mathcal
  F^{-1}\big((m_{2^n}\Xi_n-\nu_{2^n})\hat{f}\big)\big\|_{\ell^p}\lesssim 2^{-c_pn}\|f\|_{\ell^p}.
\end{align}
Next, for every $n, s\in\NN_0$ we define the multiplier
\[
  \nu_{2^n}^s(\xi)=\sum_{a/q\in\mathscr U_{(s+1)^l}\setminus\mathscr
    U_{s^l}}
G(a/q)\Phi_{2^n}(\xi-a/q) \eta_s(\xi - a/q).
\]
Notice that, by \eqref{eq:83}, if $\eta_s(\xi -a/q)-\eta_n(\xi - a/q) \neq 0$ then
\[
	|\Phi_{2^n}(\xi-a/q)|\lesssim 2^{-\chi n/d}.
\]
Thus
\begin{align*}
	\big|\nu_{2^n}(\xi)-\sum_{0\le s<n}\nu_{2^n}^s(\xi)\big| 
	&\le\sum_{0\le s<n} \sum_{a/q\in\mathscr U_{(s+1)^l}\setminus\mathscr U_{s^l}} 
	|G(a/q)| |\Phi_{2^n}(\xi-a/q)| \big|\eta_s(\xi - a/q) - \eta_n(\xi - a/q) \big| \\
	&\lesssim 2^{-\chi n/d},
\end{align*}
which combined with Plancherel's theorem implies that
\begin{align}
  \label{eq:76}
  	\big\|\mathcal F^{-1}\big(\big(\nu_{2^n}-\sum_{0\le
    s<n}\nu_{2^n}^s\big)\hat{f}\big)\big\|_{\ell^2}
	\lesssim 2^{-\chi n/d}\|f\|_{\ell^2}.
\end{align}
Moreover, since $|\mathscr U_{s^l}|\le |\mathscr U_{n^l}|\lesssim e^{(d+1)n^{1/10}}$ we
have
\begin{align}
  \label{eq:78}
  \big\|\calF^{-1}\big(\big(\nu_{2^n}-\sum_{0\le
    s<n}\nu_{2^n}^s\big)\hat{f}\big)\big\|_{\ell^p}
	\lesssim e^{(d+1)n^{1/10}}\|f\|_{\ell^p}.
\end{align}
Interpolating \eqref{eq:76} with \eqref{eq:78} one immediately
concludes that for some $c_p>0$,
\begin{align}
  \label{eq:80}
  \big\|\mathcal F^{-1}\big(\big(\nu_{2^n}-\sum_{0\le
    s<n}\nu_{2^n}^s\big)\hat{f}\big)\big\|_{\ell^p}\lesssim 2^{-c_pn}\|f\|_{\ell^p}.
\end{align}
In view of \eqref{eq:69}, \eqref{eq:73} and \eqref{eq:80} it suffices to prove that
for every $s\in\NN_0$ we have
\begin{align}
  \label{eq:81}
	\Big\|\Big(\sum_{t\in\NN}\sup_{n\in\NN_0}
	\big|\calF^{-1}\big(\nu_{2^n}^s\hat{f_t}\big)\big|^2\Big)^{1/2}\Big\|_{\ell^p}
  	\lesssim
  	(s+1)^{-2}
  	\Big\|\big(\sum_{t\in\NN}|f_t|^2\big)^{1/2}\Big\|_{\ell^p}.
\end{align}

\subsection{$\ell^2(\ZZ^d)$ estimates for \eqref{eq:81}}
Our aim is to show
\begin{align}
  \label{eq:178}
\Big\|\Big(\sum_{t\in\NN}\sup_{n\in\NN_0}\big|\mathcal
  F^{-1}\big(\nu_{2^n}^s\hat{f_t}\big)\big|^2\Big)^{1/2}\Big\|_{\ell^2}
  \lesssim
  (s+1)^{-\delta l+1}
  \Big\|\Big(\sum_{t\in\NN}|f_t|^2\Big)^{1/2}\Big\|_{\ell^2}
\end{align}
with $\delta>0$ as in \eqref{eq:20}. 
In order to do so, we only need to prove the
following theorem.
\begin{theorem}
	\label{thm:2}
	There is $C > 0$ such that for any $s \in \NN_0$ and
for every        $f\in\ell^2\big(\ZZ^d\big)$ we have
	\[
          \big\lVert
	\sup_{n \in \NN_0}
	\big\lvert
	\calF^{-1}\big(\nu_{2^n}^s \hat{f}\big)
	\big\rvert
	\big\rVert_{\ell^2}
	\leq
	C
	(s+1)^{-\delta l+1}
	\lVert f \rVert_{\ell^2}
	\]
with $l\in\NN$ defined in \eqref{eq:211} and $\delta>0$ as in \eqref{eq:20}.
\end{theorem}
\begin{proof}
	For $s \in \NN_0$ we set $\kappa_s = 20 d \big(\lfloor (s+1)^{1/10}\rfloor+1\big)$ and $Q_s = \big(\big\lfloor e^{(s+1)^{1/10}}\big\rfloor\big)!$. Firstly,
	we estimate the supremum over $0 \le n \le 2^{\kappa_s}$. By Lemma \ref{lem:6} we have
	$$
	\big\lVert
	\sup_{0\le n\le  2^{\kappa_s}}
	\big\lvert
	\calF^{-1}\big(\nu_{2^n}^s \hat{f} \big)
	\big\rvert
	\big\rVert_{\ell^2}
	\lesssim
	\big\lVert
	\calF^{-1}\big(\nu_1^s \hat{f} \big)
	\big\rVert_{\ell^2}
	+
	\sum_{i = 0}^{\kappa_s}
	\Big(
	\sum_{j = 0}^{2^{\kappa_s-i}-1}
	\Big\lVert
	\sum_{m \in I_j^i}
	\calF^{-1}\big((\nu_{2^{m+1}}^s-\nu_{2^{m}}^s) \hat{f}\big)
	\Big\rVert_{\ell^2}^2
	\Big)^{1/2}
	$$
	where $I_j^i = [j 2^i, (j+1) 2^i)$. For any $i \in \{0, \ldots, \kappa_s\}$, by Plancherel's
	theorem  we get
	\begin{align*}
		& \sum_{j = 0}^{2^{\kappa_s-i}-1}
	\Big\lVert
	\sum_{m \in I_j^i}
	\calF^{-1}\big((\nu_{2^{m+1}}^s-\nu_{2^{m}}^s) \hat{f}\big)
	\Big\rVert_{\ell^2}^2
		\leq
		\sum_{a/q \in \mathscr{U}_{(s+1)^l}\setminus\mathscr{U}_{s^l}}
		\abs{G(a/q)}^2 \\
		&\qquad\qquad \times
		\sum_{j = 0}^{2^{\kappa_s-i} - 1}
		\sum_{m, m' \in I_j^i}
		\int_{\TT^d}
		\abs{\Delta_m(\xi - a/q)}
		\cdot
		\abs{\Delta_{m'}(\xi - a/q)}
		\cdot
		\eta_s(\xi - a/q)^2
		\abs{\hat{f}(\xi)}^2
		{\: \rm d}\xi,
	\end{align*}
where $\Delta_m(\xi)=\Phi_{2^{m+1}}(\xi)-\Phi_{2^{m}}(\xi)$.
	In the last step we have used disjointness of supports of $\eta_s(\cdot - a/q)$
	while $a/q$ varies over $\mathscr{U}_{(s+1)^l}\setminus\mathscr{U}_{s^l}$.
	Using \eqref{eq:83} and \eqref{eq:84} we conclude
	\[
		\sum_{m \in \ZZ} \big\lvert
        \Delta_m(\xi)  \big\rvert \lesssim  \sum_{m \in \ZZ} 
        \min\big\{|2^{mA}\xi|_{\infty}, |2^{mA}\xi|_{\infty}^{-1/d}\big\}   \lesssim 1.
	\]
	Therefore, by \eqref{eq:20} we may estimate
        \begin{align*}
        \sum_{j = 0}^{2^{\kappa_s-i} - 1}
	\Big\lVert
	\sum_{m \in I_j^i}
	\calF^{-1}\big((\nu_{2^{m+1}}^s-\nu_{2^{m}}^s) \hat{f}\big)
	\Big\rVert_{\ell^2}^2
	& \lesssim
	(s+1)^{-2\delta l}
	\sum_{a/q \in \mathscr{U}_{(s+1)^l}\setminus\mathscr{U}_{s^l}}
	\int_{\TT^d} \eta_s(\xi - a/q)^2 \abs{\hat{f}(\xi)}^2 {\: \rm d}\xi \\
	& \lesssim
	(s+1)^{-2\delta l}
	\lVert f \rVert_{\ell^2}^2,
    \end{align*}
	because $q\ge s^l$ for $a/q \in \mathscr{U}_{(s+1)^l}\setminus\mathscr{U}_{s^l}$.  We have just proven
	\begin{equation}
		\label{eq:88}
		\big\lVert
		\sup_{0\le n \leq 2^{\kappa_s}}
		\big\lvert
		\calF^{-1}\big(\nu_{2^n}^s \hat{f} \big)
		\big\rvert
		\big\rVert_{\ell^2}
		\lesssim
		\kappa_s (s+1)^{-\delta l } 
		\lVert f \rVert_{\ell^2}
		\lesssim (s+1)^{-\delta l +1}
		\vnorm{f}_{\ell^2}.
	\end{equation}
	Next, we consider the case when the supremum is taken over $n \geq 2^{\kappa_s}$. For any
	$x, y \in \ZZ^d$ we define
	$$
	I(x, y)
	=
	\sup_{n \geq 2^{\kappa_s}}
	\Big\lvert
	\sum_{a/q \in \mathscr{U}_{(s+1)^l}\setminus\mathscr{U}_{s^l}}
	G(a/q)
	e^{-2\pi i \sprod{(a/q)}{x}}
	\calF^{-1} \big(
	\Phi_{2^n}
	\eta_s
	\hat{f}(\cdot + a/q)
	\big)(y)
	\Big\rvert
	$$
    and
	\begin{align*}
		J(x, y) =
		\sum_{a/q \in \mathscr{U}_{(s+1)^l}\setminus\mathscr{U}_{s^l}} G(a/q) e^{-2\pi i \sprod{(a/q)}{x}}
		\calF^{-1} \big( \eta_s \hat{f}(\cdot + a/q)\big)(y).
	\end{align*}
	By Plancherel's theorem, for any $u \in \NN^d_{Q_s}$ and $a/q \in \mathscr{U}_{(s+1)^l}\setminus\mathscr{U}_{s^l}$
	we have
	\begin{align*}
		& \big\lVert
		\calF^{-1}\big( \Phi_{2^n} \eta_s \hat{f}\big(\cdot + a/q\big)\big)(x+u)
		-
		\calF^{-1}\big( \Phi_{2^n} \eta_s \hat{f}\big(\cdot + a/q\big)\big)(x)
		\big\rVert_{\ell^2(x)} \\
		& \qquad\qquad =
		\big\lVert
		(1 - e^{-2\pi i \sprod{\xi}{u}}) \Phi_{2^n}(\xi) \eta_s(\xi) \hat{f}(\xi + a/q)
		\big\rVert_{L^2({\rm d}\xi)} \\
		& \qquad\qquad 
		\lesssim
		2^{-n/d}
		\cdot
		\norm{u}
		\cdot
		\big\lVert
		\eta_s(\cdot - a/q) \hat{f}
		\big\rVert_{L^2}
	\end{align*}
	since, by \eqref{eq:83}, we see
	$$
	\sup_{\xi \in \TT^d}
	\norm{\xi} \cdot \abs{\Phi_{2^n}(\xi)}
	\lesssim
	\sup_{\xi \in \TT^d}
	\norm{\xi} \cdot \norm{2^{nA}\xi}^{-1/d}
	\leq 2^{-n/d}.
	$$
	Therefore,
        \begin{align*}
        \big\lvert
		\lVert I(x, x+u) \rVert_{\ell^2(x)}
		- \lVert I(x,x) \rVert_{\ell^2(x)}
		\big\rvert
		&\lesssim
		\norm{u}
		\sum_{n = 2^{\kappa_s}}^\infty
		2^{-n/d}
		\sum_{a/q \in \mathscr{U}_{(s+1)^l}\setminus\mathscr{U}_{s^l}}
        \lVert \eta_s(\cdot - a/q) \hat{f} \rVert_{\ell^2}\\
        &\lesssim
		2^{-2^{\kappa_s}/d}\cdot Q_s\cdot |\mathscr{U}_{(s+1)^l}|\cdot\|f\|_{\ell^2}. 
   	\end{align*}
	By \eqref{eq:7} the set $\mathscr{U}_{(s+1)^l}$ contains at most $e^{(d+1)(s+1)^{1/10}}$ elements and
	\[
		 2^{\kappa_s}(\log 2)/d - (s+1)^{1/10}e^{(s+1)^{1/10}}-
        (d+1) (s+1)^{1/10} \geq  s,
	\] 
for sufficiently large $s\ge 0$. Thus we obtain
	$$
	\lVert I(x, x) \rVert_{\ell^2(x)} \lesssim
	\lVert I(x, x+u) \rVert_{\ell^2(x)} + 2^{-s } \lVert f \rVert_{\ell^2}.
	$$
	In particular,
	\[
	\big\lVert
	\sup_{n \geq 2^{\kappa_s}}
	\big\lvert
	\calF^{-1}\big(\nu_{2^n}^s \hat{f} \big)
	\big\rvert
	\Big\rVert_{\ell^2}^2
	\lesssim
	\frac{1}{Q_s^d}
	\sum_{u \in \NN_{Q_s}^d}
	\big\lVert I(x, x+u) \big\rVert_{\ell^2(x)}^2
	+
	2^{-2s }
	\lVert f \rVert_{\ell^2}^2.
	\]
	Let us observe that the functions $x \mapsto I(x, y)$ and $x \mapsto J(x, y)$ are
	$Q_s\ZZ^d$-periodic. Therefore,  by double change of variables we get
	$$
	\sum_{u \in \NN_{Q_s}^d}
	\lVert I(x, x+u) \rVert_{\ell^2(x)}^2
	=
	\sum_{x \in \ZZ^d}
	\sum_{u \in \NN_{Q_s}^d}
	I(x-u, x)^2
	=
	\sum_{x \in \ZZ^d}
	\sum_{u \in \NN_{Q_s}^d} I(u, x)^2
	=
	\sum_{u \in \NN_{Q_s}^d}
	\lVert I(u, x) \rVert_{\ell^2(x)}^2,
	$$
where in the second equality periodicity has been used. 
	Next by Proposition \ref{prop:4} (or \cite[Proposition 2.1]{MSW}) we obtain
	\begin{align*}
	\sum_{u \in \NN_{Q_s}^d}
	\lVert I(u, x) \rVert_{\ell^2(x)}^2 & \lesssim 
		\sum_{u \in \NN_{Q_s}^d}
		\lVert
		J(u, x)
		\rVert_{\ell^2(x)}^2 
		=	\sum_{u \in \NN_{Q_s}^d} \lVert J(x, x+u) \rVert_{\ell^2(x)}^2 \\
		& =\sum_{u \in \NN_{Q_s}^d}
		\int_{\TT^d}\Big|\sum_{a/q\in\mathscr{U}_{(s+1)^l}\setminus\mathscr{U}_{s^l}}G(a/q) e^{2\pi i \sprod{(a/q)}{u}}
		\eta_s(\xi-a/q)
		\hat{f}(\xi)\Big|^2
		{\: \rm d}\xi \\
		& \lesssim
		(s+1)^{-2\delta l}
		Q_s^d
		\cdot
		\|f\|_{\ell^2}^2.
	\end{align*}
	In the last step we have  used \eqref{eq:20} and the disjointness of supports of
	$\eta_s(\cdot - a/q)$ while $a/q$ varies over $\mathscr{U}_{(s+1)^l}\setminus\mathscr{U}_{s^l}$. Therefore,
	\[
		\big\lVert
		\sup_{n \geq 2^{\kappa_s}}
		\big\lvert
		\calF^{-1}\big(\nu_{2^n}^s \hat{f} \big)
		\big\rvert
		\big\rVert_{\ell^2}
		\lesssim
		(s+1)^{-\delta l} 
		\vnorm{f}_{\ell^2},
	\]
	which together with \eqref{eq:88} concludes the proof.
\end{proof}
\subsection{$\ell^p\big(\ZZ^d\big)$ estimates for \eqref{eq:81}} For
$s\in\NN_0$ let
 $$\kappa_s = 20 d \big(\lfloor (s+1)^{1/10}\rfloor+1\big)$$ 
and 
$$Q_s = \big(\big\lfloor e^{(s+1)^{1/10}}\big\rfloor\big)!$$
be as in the proof of Theorem \ref{thm:2}. We show that  for every
$p\in(1, \infty)$ there is a constant $C_p>0$ such that for every
$s\in\NN_0$ and all sequences $(f_t: t\in\NN) \in \ell^p\big(\ell^2\big(\ZZ^{d}\big)\big)$ we have
\begin{align}
  \label{eq:89}
  \Big\|\Big(\sum_{t\in\NN}\sup_{n\in\NN_0}\big|\calF^{-1}\big(\nu_{2^n}^s\hat{f_t}\big)\big|^2\Big)^{1/2}
	\Big\|_{\ell^p}
  \le C_p
  s\log(s+2)
  \Big\|\Big(\sum_{t\in\NN}|f_t|^2\Big)^{1/2}\Big\|_{\ell^p}.
\end{align}
Then interpolation between \eqref{eq:178} and \eqref{eq:89}  will immediately
imply \eqref{eq:81}. The proof of \eqref{eq:89} consists of two
steps. We separately bound the supremum when $0\le n\le 2^{\kappa_s}$
and when $n\ge 2^{\kappa_s}$, see Theorem \ref{thm:7} and Theorem
\ref{thm:8}, respectively. 

\begin{theorem}
\label{thm:7}
Let $p\in(1, \infty)$ then there is a constant $C_p>0$ such that for any
$s\in\NN_0$ and for every sequence $(f_t: t\in\NN) \in
\ell^p\big(\ell^2\big(\ZZ^{d}\big)\big)$ we have
\[
  \Big\|\Big(\sum_{t\in\NN}\sup_{0\le n\le 2^{\kappa_s}}\big|\mathcal
  F^{-1}\big(\nu_{2^n}^s\hat{f_t}\big)\big|^2\Big)^{1/2}\Big\|_{\ell^p}
  \le C_p
  s\log(s+2)
  \Big\|\Big(\sum_{t\in\NN}|f_t|^2\Big)^{1/2}\Big\|_{\ell^p}.
\]
\end{theorem}
\begin{proof}
	We set $J=\big\lfloor e^{(s+1)^{1/2}}\big\rfloor$ and
        define 
	$$
	\mu_J(\xi) = \hat{K}_J(\xi)= J^{-k} \sum_{y \in \NN_J^k} e^{2\pi i \sprod{\xi}{\calQ(y)}},
	$$
        where $K_J$ is the kernel of the operator $M_J$, see
        \eqref{eq:82}.  We notice that for each $r \in [1, \infty]$ we have
	\[
		\big\lVert
		\calF^{-1} \big(\mu_J \hat{f} \big)
		\big\rVert_{\ell^r}
		\leq
		\vnorm{f}_{\ell^r}.
	\]
	Moreover, if $\xi \in \TT^d$ is such that  $\norm{\xi_{\gamma} - a_{\gamma}/q}
        \leq 2^{-s(|\gamma|-\chi)}$  for every $\gamma\in\Gamma$ with some $1 \leq q \leq e^{(s+1)^{1/10}}$ and
	$a \in A_q$, then 
	\begin{equation*}
		\mu_J(\xi) = G(a/q)\Phi_J(\xi-a/q)+\mathcal O\big(e^{-\frac{1}{2}(s+1)^{1/2}}\big).
	\end{equation*}
	Indeed, by Proposition \ref{pr:0} with $L_1=2^s$,
        $L_2=2^{s\chi}$, $L_3=e^{(s+1)^{1/10}}$ and $N=J$ we see that the
        error term is dominated by
        \begin{align*}
		L_3J^{-1}+L_2L_3J^{-1}\sum_{\gamma \in
		\Gamma}\big(J/L_1\big)^{|\gamma|}
		& \lesssim
		e^{(s+1)^{1/10}-(s+1)^{1/2}}+2^{s\chi}e^{(s+1)^{1/10}-(s+1)^{1/2}}(e^{(s+1)^{1/2}}\cdot2^{-s}) \\
		& \lesssim 
		e^{-\frac{1}{2}(s+1)^{1/2}}.
        \end{align*}
		Therefore,
	\begin{equation}
		\label{eq:95}
		\big\lvert \mu_J(\xi) - G(a/q) \big\rvert 
		\lesssim\big|G(a/q)(\Phi_J(\xi-a/q)-1)\big|+e^{-\frac{1}{2}(s+1)^{1/2}}
		\lesssim e^{-\frac{1}{2}(s+1)^{1/2}},
	\end{equation}
	because
	\[
		\big|\Phi_J(\xi-a/q)-1\big|
		\lesssim 
		 \big\lvert J^{A}(\xi-a/q) \big\rvert\lesssim e^{(s+1)^{1/2}}2^{-s(1-\chi)}.
	\]
	Next, let us define the multipliers 
	\[
  	\Pi_{2^n}^s(\xi)=\sum_{a/q\in\mathscr U_{(s+1)^l}\setminus\mathscr
    U_{s^l}}
	\Phi_{2^n}(\xi-a/q)\eta_s(\xi - a/q).
	\]
	We observe that by \eqref{eq:95} we have
	\[
  		\nu_{2^n}^s(\xi)-\mu_J(\xi)\Pi_{2^n}^s(\xi)=\mathcal O\big(e^{-\frac{1}{2}(s+1)^{1/2}}\big).
	\]
	Hence, by Plancherel's theorem we get 
\begin{align}
  \label{eq:99}
  \big\|\mathcal
  F^{-1}\big((\nu_{2^n}^s-\mu_J\Pi_{2^n}^s)\hat{f}\big)\big\|_{\ell^2}\lesssim e^{-\frac{1}{2}(s+1)^{1/2}}\|f\|_{\ell^2}.
\end{align}
Moreover, by \eqref{eq:7}, for every $p\in(1, \infty)$, we have a trivial bound
\begin{align}
	\nonumber
	\big\|\mathcal F^{-1}\big((\nu_{2^n}^s-\mu_J\Pi_{2^n}^s)\hat{f}\big)\big\|_{\ell^p}
	&\lesssim
	|\mathscr U_{(s+1)^l}| \cdot \|f\|_{\ell^p} \\
	\label{eq:100}
	&\lesssim e^{(d+1)(s+1)^{1/10}}\|f\|_{\ell^p}.  
\end{align}
Interpolating now \eqref{eq:99} with \eqref{eq:100} one has that for some
$c_p>0$,
\begin{align}
  \label{eq:101}
\big\|\mathcal
  F^{-1}\big((\nu_{2^n}^s-\mu_J\Pi_{2^n}^s)\hat{f}\big)\big\|_{\ell^p}\lesssim
   e^{-c_p(s+1)^{1/2}}\|f\|_{\ell^p}.  
\end{align}
Thus in view of \eqref{eq:69} and \eqref{eq:101} we obtain
\begin{align*}
  \Big\|\Big(\sum_{t\in\NN}\sup_{0\le n\le 2^{\kappa_s}}
	\big|\calF^{-1}\big((\nu_{2^n}^s-\mu_J\Pi_{2^n}^s)\hat{f_t}\big)\big|^2\Big)^{1/2}\Big\|_{\ell^p}
	& \lesssim \sum_{n=0}^{2^{\kappa_s}}  \Big\|\Big(\sum_{t\in\NN}
	\big|\calF^{-1}\big((\nu_{2^n}^s-\mu_J\Pi_{2^n}^s)\hat{f_t}\big)\big|^2\Big)^{1/2}\Big\|_{\ell^p} \\
	& \lesssim 
	\Big\|\big(\sum_{t\in\NN}|f_t|^2\big)^{1/2}\Big\|_{\ell^p} 
\end{align*}
since $2^{\kappa_s}e^{-c_p(s+1)^{1/2}}\lesssim 1$. The proof of
Theorem \ref{thm:7} will be completed if we show 
	\[
	\Big\|\Big(\sum_{t\in\NN}\sup_{0\le n\le 2^{\kappa_s}}
	\big|\calF^{-1}\big(\Pi_{2^n}^s\hat{f_t}\big)\big|^2\Big)^{1/2}\Big\|_{\ell^p} 
	\lesssim \kappa_{s} \log(s+2)
 	\Big\|\big(\sum_{t\in\NN}|f_t|^2\big)^{1/2}\Big\|_{\ell^p}.
	\]
  Appealing to inequality \eqref{eq:21} we see that
\begin{align*}
	\Big\|\Big(\sum_{t\in\NN}\sup_{0\le n\le 2^{\kappa_s}}
	\big|\calF^{-1}\big(\Pi_{2^n}^s\hat{f_t}\big)\big|^2\Big)^{1/2}\Big\|_{\ell^p}
	&\lesssim \Big\|\Big(\sum_{t\in\NN}\big|\mathcal
  F^{-1}\big(\Pi_{1}^s\hat{f_t}\big)\big|^2\Big)^{1/2}\bigg\|_{\ell^p}\\
	&\phantom{\lesssim}+\sum_{i=0}^{\kappa_s}\bigg\|\Big(\sum_{t\in\NN}\sum_{j = 0}^{2^{\kappa_s-i}-1}
	\Big|
	\sum_{m \in I_j^i}
	\calF^{-1}\big((\Pi_{2^{m+1}}^s-\Pi_{2^{m}}^s) \hat{f_t}\big)
	\Big|^2
	\Big)^{1/2}\bigg\|_{\ell^p},
\end{align*}
where $I_j^i = [j 2^i, (j+1)2^i)$.
In view of Lemma \ref{lem:30} if suffices to show that for every 
$i\in\{0, 1, \ldots, \kappa_s\}$ and $\omega\in[0, 1]$ we have 
\begin{align}
  \label{eq:104}
\Big\|\sum_{j = 0}^{2^{\kappa_s-i}-1}
	\sum_{m \in I_j^i}\varepsilon_j(\omega)
	\calF^{-1}\big((\Pi_{2^{m+1}}^s-\Pi_{2^{m}}^s) \hat{f}\big)
	\Big\|_{\ell^p}\lesssim \log(s+2) \|f\|_{\ell^p}
\end{align}
for any Rademacher sequence  
	$\varepsilon=\big(\varepsilon_j(\omega) : 0 \leq j <
        2^{\kappa_s - i} \big)$ with $\varepsilon_j(\omega) \in \{-1,
        1\}$. To do so, let us consider the operator
	$$
	\calT_\varepsilon f 
	= \sum_{a/q \in \mathscr{U}_{(s+1)^l}\setminus \mathscr{U}_{s^l}} \calF^{-1}
	\big(\Theta(\cdot- a/q) \eta_s(\cdot- a/q) \hat{f}\big)
	$$
	with 
	$$
	\Theta
	= \sum_{j = 0}^{2^{\kappa_s - i}-1} \varepsilon_j(\omega)
	\sum_{m \in I_j^i} (\Phi_{2^{m+1}}-\Phi_{2^m}).
	$$
	We notice that the multiplier
        $\Theta$ corresponds to a continuous singular Radon transform.
        Thus $\Theta$ defines a bounded
	operator on $L^r\big(\RR^d\big)$ for any $r \in (1, \infty)$ with the bound independent of
	the underlying  sequence $\big(\varepsilon_j(\omega) : 0 \leq j \leq 2^{\kappa_s-i}\big)$
	(see \cite[Section 11]{bigs} or \cite{DuoRdF}). Hence, by Theorem \ref{th:3} we get
	\[
	\|\calT_\varepsilon f\|_{\ell^p}\lesssim \log(s+2) \|f\|_{\ell^p}.
	\]
	Consequently, we obtain \eqref{eq:104} and the proof of Theorem \ref{thm:7} is completed. 
\end{proof}
For each $N \in \NN$ and $s\in\NN_0$ we define the  multiplier
\[
	\Omega_N^s(\xi)
	=\sum_{a/q \in \mathscr{U}_{(s+1)^l}\setminus \mathscr{U}_{s^l}}
	G(a/q) \Theta_N(\xi - a/q) \vrho_s(\xi - a/q),
\]
where $\vrho_s(\xi)=\eta\big(Q_{s+1}^{3dA}\xi\big)$ and
$(\Theta_N: N\in\NN)$ is a sequence of multipliers on $\RR^d$
obeying \eqref{eq:54}.
\begin{theorem}
	\label{th:1}
	Let $p \in (1, \infty)$ then there exists $C_{p}>0$ such that for any
	$s \in\NN_0$ and for every sequence $(f_t:
        t\in\NN)\in\ell^p\big(\ell^2\big(\ZZ^d\big)\big)$ we have
	\[
          \Big\|\Big(\sum_{t\in\NN}\big|\mathcal
          N\big(\calF^{-1}\big(\Omega_N^s \hat{f_t}\big):
          N\in\NN\big)\big|^2\Big)^{1/2}\Big\|_{\ell^p}\le C_p\mathbf
          B_p \log(s+2)
			\Big\|\big(\sum_{t\in\NN}|f_t|^2\big)^{1/2}\Big\|_{\ell^p}.
	\]
\end{theorem}

\begin{proof}
	Let us observe that
	$$
	\calF^{-1}\big(\Theta_N (\cdot - a/q) \vrho_s(\cdot - a/q) \hat{f} \big)(Q_s x + m)
	=
	\calF^{-1}\big(\Theta_N \vrho_s \hat{f}(\cdot  + a/q)\big)(Q_sx + m) e^{-2\pi i \sprod{(a/q)}{m}}.
	$$
	Therefore,
        \begin{align*}
          \bigg\|\Big(\sum_{t\in\NN}\big|\mathcal
          N &\big(\calF^{-1}\big(\Omega_N^s \hat{f_t}\big):
          N\in\NN\big)\big|^2\Big)^{1/2}\bigg\|_{\ell^p}^p\\
	=&
	\sum_{m \in \NN_{Q_s}^d}
\bigg\|\Big(\sum_{t\in\NN}\big|\mathcal N\big(\calF^{-1}\big(\Theta_N \vrho_s F_t(\cdot\ ;
        m) \big)(Q_s x + m): N \in \NN\big)
	\big|^2\Big)^{1/2}\bigg\|_{\ell^p(x)}^p,
        \end{align*}
	where
	\begin{equation}
		\label{eq:108}
		F_t(\xi; m)
		= \sum_{a/q \in \mathscr{U}_{(s+1)^l}\setminus \mathscr{U}_{s^l}} G(a/q) \hat{f_t}(\xi + a/q) e^{-2\pi i \sprod{(a/q)}{m}}.
	\end{equation}
	Now, by Proposition \ref{prop:4} and  \eqref{eq:108} we get
	\begin{align*}
		& \sum_{m \in \NN_{Q_s}^d}
\Big\|\Big(\sum_{t\in\NN}\big|\mathcal N\big(\calF^{-1}\big(\Theta_N \vrho_s F_t(\cdot\ ;
        m) \big)(Q_s x + m): N \in \NN\big)
	\big|^2\Big)^{1/2}\Big\|_{\ell^p(x)}^p\\
		& \qquad \qquad \le C_p^p\mathbf B_p^p
		\sum_{m \in \NN_{Q_s}^d}
\Big\|\Big(\sum_{t\in\NN}\big|\calF^{-1}\big(\vrho_s F_t(\cdot\ ;
        m) \big)(Q_s x + m)
	\big|^2\Big)^{1/2}\Big\|_{\ell^p(x)}^p\\
&\qquad \qquad =C_p^p\mathbf B_p^p
		\Big\lVert\Big(\sum_{t\in\NN}\big|
		\sum_{a/q \in \mathscr{U}_{(s+1)^l}\setminus \mathscr{U}_{s^l}} G(a/q) \calF^{-1}\big(\vrho_s(\cdot - a/q) \hat{f_t}\big)\big|^2\Big)^{1/2}
		\Big\rVert_{\ell^p}^p.
	\end{align*}
In view of \eqref{eq:69} the proof will be completed if we show that
\begin{align}
  \label{eq:79}
  \big\|\mathcal
  F^{-1}\big(\tilde{\Pi}_s^G\hat{f}\big)\big\|_{\ell^p}\lesssim \log(s+2)\|f\|_{\ell^p},
\end{align}
where
\[
  \tilde{\Pi}_s^G(\xi)=\sum_{a/q \in \mathscr{U}_{(s+1)^l}\setminus \mathscr{U}_{s^l}} G(a/q) \vrho_s(\xi - a/q). 
\]
Arguing in a similar way as in the proof of Theorem
\ref{thm:7}, by \eqref{eq:95} we obtain
\begin{align}
  \label{eq:94}
  \big|\tilde{\Pi}_s^G(\xi)-\mu_J(\xi)\tilde{\Pi}_s(\xi)\big|
	\lesssim e^{-\frac{1}{2}(s+1)^{1/2}},
\end{align}
where
\[
  \tilde{\Pi}_s(\xi)=\sum_{a/q \in \mathscr{U}_{(s+1)^l}\setminus \mathscr{U}_{s^l}} \vrho_s(\xi - a/q). 
\]
Therefore, \eqref{eq:94} combined with Plancherel's theorem yields
\begin{align}
\label{eq:112}
  \big\|\mathcal
  F^{-1}\big((\tilde{\Pi}_s^G-\mu_J\tilde{\Pi}_s)\hat{f}\big)\big\|_{\ell^2}\lesssim e^{-\frac{1}{2}(s+1)^{1/2}}\|f\|_{\ell^2}.
\end{align}
We can conclude by interpolation with \eqref{eq:112} that
	\[
  \big\|\mathcal
  F^{-1}\big((\tilde{\Pi}_s^G-\mu_J\tilde{\Pi}_s)\hat{f}\big)\big\|_{\ell^p}\lesssim \|f\|_{\ell^p},
\]
since by Theorem \ref{th:3} we have
\[ 
 \big\|\mathcal
  F^{-1}\big(\tilde{\Pi}_s\hat{f}\big)\big\|_{\ell^p}\lesssim \log(s+2) \|f\|_{\ell^p} 
\]
and by the trivial bound, due to \eqref{eq:7}, we have 
\[
  \big\|\mathcal
  F^{-1}\big(\tilde{\Pi}_s^G\hat{f}\big)\big\|_{\ell^p}\lesssim
  |\mathscr{U}_{(s+1)^l}| \cdot \|f\|_{\ell^p}
\lesssim e^{(d+1)(s+1)^{1/10}}\|f\|_{\ell^p}.
\]
This establishes the bound in \eqref{eq:79} and the proof of Theorem
\ref{th:1} is finished. 
\end{proof}

\begin{theorem}
\label{thm:8}
Let $p\in(1, \infty)$ then there is a constant $C_p>0$ such that for any
$s\in\NN_0$ and for every sequence $(f_t: t\in\NN)\in\ell^p\big(\ell^2\big(\ZZ^d\big)\big)$ we have 
\[  
\Big\|\Big(\sum_{t\in\NN}\sup_{n\ge 2^{\kappa_s}}\big|\mathcal
  F^{-1}\big(\nu_{2^n}^s\hat{f_t}\big)\big|^2\Big)^{1/2}\Big\|_{\ell^p}
  \le C_p \log(s+2)
  \Big\|\big(\sum_{t\in\NN}|f_t|^2\big)^{1/2}\Big\|_{\ell^p}.
\]
\begin{proof}
  Theorem \ref{thm:17} guarantees that the sequence $(\Phi_{2^n}:
  n\in\NN_0)$ satisfies \eqref{eq:54}.  Thus in view of Theorem
  \ref{th:1} with $\Theta_N=\Phi_N$ and \eqref{eq:69} it suffices to prove that for any
  $n\ge 2^{\kappa_s}$ we have
\begin{align}
  \label{eq:125}
  \big\|\mathcal
  F^{-1}\big((\nu_{2^n}^s-\Omega_{2^n}^s)\hat{f}\big)\big\|_{\ell^p}\lesssim 2^{-c_p^1n}e^{c_p^2(s+1)^{1/10}}\|f\|_{\ell^p}
\end{align}
for some $c_p^1, c_p^2>0$. Obviously we have 
\begin{align}
	\nonumber
  	\big\|\mathcal F^{-1}\big((\nu_{2^n}^s-\Omega_{2^n}^s)\hat{f}\big)\big\|_{\ell^p}
	&\lesssim
	\abs{\mathscr{U}_{(s+1)^l}} \cdot \|f\|_{\ell^p} \\
	\label{eq:126}
	&\lesssim e^{(d+1)(s+1)^{1/10}}\|f\|_{\ell^p}.
\end{align}
  	Next, observe that $\vrho_s(\xi - a/q) - \eta_s(\xi - a/q) \neq 0$ implies that 
	$\abs{\xi_\gamma - a_\gamma/q} \geq (16 d)^{-1} Q_{s+1}^{-3 d\abs{\gamma}}$ for some
	$\gamma \in \Gamma$. Therefore, for $n \geq 2^{\kappa_s}$ we have
	$$
	2^{n\abs{\gamma}} \cdot \big\lvert \xi_{\gamma} - a_{\gamma}/q \big\rvert
	\gtrsim
	2^{n|\gamma|} Q_{s+1}^{-3d|\gamma|} \gtrsim 2^{n/2},
	$$
since 
\[
2^{n/2}Q_{s+1}^{-3d}\ge
2^{2^{\kappa_s-1}}e^{-3d(s+1)^{1/10}e^{(s+1)^{1/10}}}\ge e^{(s+1)^{1/10}}
\]
for sufficiently large $s\in\NN_0$. Thus using \eqref{eq:83}, we obtain
	$$
	\sabs{\Phi_{2^n}(\xi - a/q)} \lesssim 2^{-n/(2d)}.
	$$
	Hence, by \eqref{eq:20},
	\[
		\Big\lvert
		\sum_{a/q \in \mathscr{U}_{(s+1)^l}\setminus \mathscr{U}_{s^l}}
		G(a/q)
		\Phi_{2^n}(\xi - a/q) \big(\eta_s(\xi - a/q) - \vrho_s(\xi - a/q) \big)
		\Big\rvert
		\leq
		C
		(s+1)^{-\delta l}
		2^{-n/(2d)}.
	\]
	By Plancherel's theorem we conclude
        \begin{align}
          \label{eq:149}
          \big\|\mathcal
  F^{-1}\big((\nu_{2^n}^s-\Omega_{2^n}^s)\hat{f}\big)\big\|_{\ell^2}\lesssim
  2^{-n/(2d)}(s+1)^{-\delta l}\|f\|_{\ell^2}.
        \end{align}
Interpolating now \eqref{eq:149} with \eqref{eq:126} we obtain
\eqref{eq:125}, which completes the proof of Theorem \ref{thm:8}.
\end{proof}
\end{theorem}

 \section{Vector-valued maximal estimates for truncated singular
   integral operators}
\label{sec:7}

The purpose of this section is to prove Theorem \ref{thm:11}. In order to do so, in view
of Lemma \ref{lem:1}, it suffices to
consider the operators $T_N^\calQ$, where $\calQ$ is the canonical polynomial defined in Section \ref{sec:5}.
Let us fix necessary notation. Given a kernel $K$ satisfying \eqref{eq:3} and \eqref{eq:4} there are functions 
$(\seq{K_j}{j \in\ZZ})$ and a constant $C > 0$ such that for $x \neq 0$ we have
\begin{align}
  \label{eq:8}
K(x) = \sum_{j \in\ZZ} K_j(x),  
\end{align}
where for each $j \in \ZZ$ the kernel $K_j$ is supported inside $2^{j-2} \leq \norm{x} \leq 2^j$,
satisfies
\[
	\norm{x}^k \abs{K_j(x)} + \norm{x}^{k+1} \norm{\nabla K_j(x)} \leq C
\]
for all $x \in \RR^k$, and has integral
$0$ (see \cite[Chapter 6, \S4.5, Chapter 13,
\S5.3]{bigs}). In our case, the summation in \eqref{eq:8} will be restricted to $j\ge0$, since $x\in\ZZ^k\setminus\{0\}$.

Next, we define a sequence $(\seq{m_j}{j \geq 0})$ of functions on $\TT^d$ by
$$
m_j(\xi) = \sum_{y \in \ZZ^k} e^{2\pi i \sprod{\xi}{\calQ(y)}} K_j(y).
$$
Then for any finitely supported function  $f$ on $\ZZ^d$ we write
$$
\calF^{-1} \big(m_j \hat{f}\big)(x) = \sum_{y \in \ZZ^k} f(x - \calQ(y)) K_j(y).
$$
For $j \geq 0$ and $\xi\in\RR^d$ we set
$$
\Phi_j(\xi) = \int_{\RR^k} e^{2\pi i \sprod{\xi}{\calQ(y)}} K_j(y) {\: \rm d}y.
$$
Using multi-dimensional version of van der Corput's lemma (see \cite[Propositon 2.1]{sw})
we obtain
\begin{equation}
	\label{eq:18}
	\lvert \Phi_j(\xi) \rvert
	\lesssim
	\min\big\{1, \norm{2^{j A} \xi}_{\infty}\big\}^{-1/d},
\end{equation}
where $A$ is the $d\times d$ diagonal matrix defined in \eqref{eq:25}. Moreover, if $j \geq 1$ then we have
\begin{equation}
	\label{eq:19}
	\lvert \Phi_j(\xi) \rvert
	=
	\Big\lvert
	\Phi_j(\xi) - \int_{\RR^k} K_j(y) {\: \rm d}y
	\Big\rvert
	\lesssim
	\min\big\{1, \norm{2^{jA} \xi}_{\infty} \big\}.
\end{equation}
The estimates in \eqref{eq:18} and \eqref{eq:19} are analogues of \eqref{eq:83} and \eqref{eq:84} respectively. 
Finally, let $\Psi_n = \sum_{j = 0}^n \Phi_j$.

Our aim is to prove  the  following dyadic version of Theorem \ref{thm:11}.
\begin{theorem}
  \label{thm:12}
	For every $p\in(1, \infty)$  there is $C_{p} > 0$ such that for all sequences
	$\big(f_t: t\in\NN\big) \in \ell^p\big(\ell^2\big(\ZZ^{d}\big)\big)$ we have
	\[
      \Big\lVert\Big(
	\sum_{t\in\NN}\sup_{n\in\NN_0}\Big|\sum_{j=0}^n\calF^{-1} \big(m_j \hat{f_t}\big)\Big|^2\Big)^{1/2}
	\Big\rVert_{\ell^p}\le
	C_{p}\Big\|\big(\sum_{t\in\NN}|f_t|^2\big)^{1/2}\Big\|_{\ell^p}.
	\]
\end{theorem}
It is easy to see that Theorem \ref{thm:12} combined with Theorem
\ref{thm:10} implies Theorem \ref{thm:11}. Indeed, for any $f\ge0$ we have the following
pointwise bound
\[
\sup_{N\in\NN}\big\lvert
	T_N f(x) 
	\big\rvert
	\lesssim\sup_{n\in\NN_0}
	\Big\lvert
	\sum_{j = 0}^n \calF^{-1}\big( m_j \hat{f} \big)(x)
	\Big\rvert
	+
\sup_{N\in\NN}\big\lvert
	M_N f(x) 
	\big\rvert.
        \]
        
        The strategy of the proof of Theorem \ref{thm:12} is much the
        same as for the proof of the result for averaging operators
        from Section \ref{sec:4}. However, there are some changes
        which could cause some confusions.  Therefore, for the
        convenience of the reader we indicate the places where the
        changes have to be done.

As in Section \ref{sec:4}, with the aid of \eqref{eq:9}, we write
\begin{equation}
\label{eq:165}
\begin{aligned}
	\Big\|
	\Big(\sum_{t\in\NN}\sup_{n\in\NN_0}\Big|\sum_{j=0}^n\calF^{-1}\big(m_j\hat{f}_t\big)\Big|^2\Big)^{1/2}
	\Big\|_{\ell^p}
	 &\le
	\Big\|\Big(\sum_{t\in\NN}\sup_{n\in\NN_0}\Big|\sum_{j=0}^n\mathcal
  	F^{-1}\big(m_{j}(1-\Xi_{j})\hat{f_t}\big)\Big|^2\Big)^{1/2}\Big\|_{\ell^p}\\
	&\phantom{\le}+\Big\|\Big(\sum_{t\in\NN}\sup_{n\in\NN_0}\Big|\sum_{j=0}^n\mathcal
  	F^{-1}\big(m_{j}\Xi_{j}\hat{f_t}\big)\Big|^2\Big)^{1/2}\Big\|_{\ell^p}.
\end{aligned}
\end{equation}

\subsection{The estimate for the first term in \eqref{eq:165}}
We see that
\[
	\Big\|
	\Big(\sum_{t\in\NN}\sup_{n\in\NN_0}\Big|\sum_{j=0}^n\calF^{-1}\big(m_j(1-\Xi_{j})\hat{f}_t\big)\Big|^2\Big)^{1/2}
	\Big\|_{\ell^p}
	\le
	\sum_{j\in\NN_0}
	\Big\|\Big(\sum_{t\in\NN}\big|\calF^{-1}\big(m_{j}(1-\Xi_{j})\hat{f_t}\big)\big|^2\Big)^{1/2}
	\Big\|_{\ell^p}.
\]
Arguing in a similar way as in the proof of inequality \eqref{eq:10},
invoking Theorem \ref{thm:3}, we can show that
\[
  \Big\|\Big(\sum_{t\in\NN}\big|\mathcal
  F^{-1}\big(m_{j}(1-\Xi_{j})\hat{f_t}\big)\big|^2\Big)^{1/2}\Big\|_{\ell^p}\le
  (j+1)^{-2}
  \Big\|\big(\sum_{t\in\NN}|f_t|^2\big)^{1/2}\Big\|_{\ell^p}.
\]
\subsection{The estimate for the second term in \eqref{eq:165}}
Proceeding as in the proof of  Proposition \ref{pr:0} (see also  \cite[Proposition 3.2]{mir1}),
we establish an approximation formula  for the multipliers $m_j$.
\begin{proposition}
  \label{prop:1}
 There is a constant $C>0$ such
  that for every $j\in\NN$ and for every $\xi\in [-1/2, 1/2)^d$ satisfying 
        $$
	\Big\lvert \xi_\gamma - \frac{a_\gamma}{q} \Big\rvert \leq
        L_1^{-|\gamma|}L_2
	$$
	for all $\gamma \in \Gamma$, where  $1\le q\le L_3\le 2^{j/2}$, $a\in
        A_q$, $L_1\ge 2^j$  and $L_2\ge1$ we have
	\[ 
          \big|m_j(\xi)-G(a/q)\Phi_{j}(\xi-a/q)\big|\le 
          C\Big(L_32^{-j}
          +L_2L_32^{-j}\sum_{\gamma \in
            \Gamma}\big(2^j/L_1\big)^{|\gamma|}\Big)
	\]
\end{proposition}
Next, for every $j\in\NN_0$, let us define the multiplier
\[
  \nu_{j}(\xi)=\sum_{a/q\in\mathscr U_{j^l}}G(a/q)\Phi_{j}(\xi-a/q)\eta_j(\xi - a/q).
\]
In a similar way as in  \eqref{eq:73} one  can conclude
that for some $c_p>0$ we have
\begin{align}
  \label{eq:170}
  \big\|\mathcal
  F^{-1}\big((m_{j}\Xi_j-\nu_{j})\hat{f}\big)\big\|_{\ell^p}\lesssim 2^{-c_pj}\|f\|_{\ell^p}.
\end{align}
For every $j, s\in\NN_0$, we introduce the multiplier
\[
	\nu_{j}^s(\xi)=\sum_{a/q\in\mathscr U_{(s+1)^l}\setminus\mathscr U_{s^l}}
	G(a/q)\Phi_{j}(\xi-a/q)\eta_s(\xi-a/q).
\]
Arguing as in the proof of the estimate \eqref{eq:80}  we obtain that for some $c_p>0$,
\begin{align}
  \label{eq:172}
  \big\|\mathcal F^{-1}\big(\big(\nu_{j}-\sum_{0\le
    s<j}\nu_{j}^s\big)\hat{f}\big)\big\|_{\ell^p}\lesssim 2^{-c_pj}\|f\|_{\ell^p}.
\end{align}
Therefore, in view of \eqref{eq:69}, \eqref{eq:170} and \eqref{eq:172} it suffices to prove that
for every $s\in\NN_0$ we have
\begin{align*}
\Big\|\Big(\sum_{t\in\NN}\sup_{n\in\NN_0}\Big|\sum_{j=0}^n\mathcal
  F^{-1}\big(\nu_{j}^s\hat{f_t}\big)\Big|^2\Big)^{1/2}\Big\|_{\ell^p}
  \lesssim
  (s+1)^{-2}
  \Big\|\big(\sum_{t\in\NN}|f_t|^2\big)^{1/2}\Big\|_{\ell^p}.
\end{align*}
Here, the line of reasoning follows parallel to the proof of the inequality \eqref{eq:81}. 
This completes the sketch of the proof of Theorem \ref{thm:12}.

\appendix
\section{Vector-valued estimates for the continues Radon operators}
\label{sec:8}
This section is devoted to provide some vector-valued estimates for maximal operators of Radon type 
in the continuous settings. To fix
notation let
$\calP=(\calP_1,\ldots, \calP_{d_0}): \RR^k \rightarrow \RR^{d_0}$ be
a polynomial mapping whose components $\calP_j$ are
polynomials with real coefficients on $\RR^k$ such that $\calP_j(0) = 0$. One of the main
objects of our interest will be
\[ 
  \mathcal M_r^{\calP}f(x)=\frac{1}{|B_r|}\int_{B_r}f\big(x-\mathcal P(y)\big) {\: \rm d}y, \quad\text{for}\quad x\in\RR^{d_0},
\]
where $B_r$ is the Euclidean ball  centered at the origin with radius $r>0$.
We prove the following result.
\begin{theorem}
\label{thm:17}
  Assume that $p\in(1, \infty)$ then there is a constant $C_p>0$ such
  that for every sequence $(f_t: t\in\NN)\in
  L^p\big(\ell^2\big(\RR^{d_0}\big)\big)$ we have
\[
    \Big\|\Big(\sum_{t\in\NN}
	\sup_{r>0} \big|\mathcal M_r^{\calP}f_t\big|^2\Big)^{1/2}\Big\|_{L^p}
	\le C_p
    \Big\|\big(\sum_{t\in\NN}|f_t\big|^2\big)^{1/2}\Big\|_{L^p}.
\]
Moreover, the implied constant is independent of the coefficients of
the polynomial mapping $\calP$.
\end{theorem}

Suppose that $K\in\mathcal C^1\big(\RR^k\setminus\{0\}\big)$ is a
Calder\'on--Zygmund kernel satisfying the differential inequality 
\begin{align}
  \label{eq:29}
  |y|^k|K(y)|+|y|^{k+1}|\nabla K(y)|\le1
\end{align}
for all $y\in\RR^k\setminus\{0\}$ and the cancellation condition
\begin{align}
	\label{eq:30}
	\sup_{0<r<R<\infty}\bigg|\int_{r\le |y|\le R}K(y){\: \rm d}y\bigg|\le 1.
\end{align}
We are also interested in truncated singular Radon transforms
\[
  \mathcal T_r^{\calP}f(x)=\int_{|y|>r}f\big(x-\mathcal P(y)\big)K(y) {\: \rm d}y, \quad\text{for}\quad x\in\RR^{d_0}.
\]
 The second main result is the inequality in Theorem \ref{thm:18}.
\begin{theorem}
\label{thm:18}
  Assume that $p\in(1, \infty)$ then there is a constant $C_p>0$ such
  that for every sequence $(f_t: t\in\NN)\in
  L^p\big(\ell^2\big(\RR^{d_0}\big)\big)$ we have
	\[
    \Big\|\Big(\sum_{t\in\NN}\sup_{r>0}\big|\mathcal
    T_r^{\calP}f_t\big|^2\Big)^{1/2}\Big\|_{L^p}\le C_p
    \Big\|\big(\sum_{t\in\NN}|f_t\big|^2\big)^{1/2}\Big\|_{L^p}.
	\]
Moreover, the implied constant is independent of the coefficients of
the polynomial mapping $\calP$.
\end{theorem}
In view of \cite{deL} or \cite[Section 11]{bigs}, we can reduce the matters in Theorem \ref{thm:17}
and Theorem \ref{thm:18} to the canonical polynomial mapping $\calQ$, where for some set of multi-indices
$\Gamma\subseteq \NN^k\setminus\{0\}$,
\[
	\calQ = (\seq{\calQ_\gamma}{\gamma \in \Gamma}) : \RR^k \rightarrow \RR^d,
\]
with $\calQ_\gamma(x) = x^\gamma$ and $x^\gamma=x_1^{\gamma_1}\cdot\ldots\cdot x_k^{\gamma_k}$. Let $d$ be the cardinality
of the set $\Gamma$. We identify $\RR^d$ with the space of all vectors whose coordinates are labeled by multi-indices
$\gamma \in \Gamma$. From now on we work with $\calP=\calQ$. To simplify notation we set
$\mathcal M_r=\mathcal M_r^{\calQ}$ and $\mathcal T_t=\mathcal T_r^{\calQ}$.

\subsection{Proof of Theorem \ref{thm:17}} 
It suffices to show that for every $p\in(1, \infty)$ there is a constant $C_p>0$ such
  that for every $(f_t: t\in\NN)\in
  L^p\big(\ell^2\big(\RR^{d}\big)\big)$ we have
\[
    \Big\|\Big(\sum_{t\in\NN}\sup_{n\in\ZZ}\big|\mathcal
    M_{2^n}f_t\big|^2\Big)^{1/2}\Big\|_{L^p}\le C_p
    \Big\|\big(\sum_{t\in\NN}|f_t\big|^2\big)^{1/2}\Big\|_{L^p}.
\]
For this purpose let $\Phi$ be a compactly smooth function on $\RR^d$ such that
\begin{align}
  \label{eq:11}
\int_{\RR^d}\Phi(y){\: \rm d}y=1.
\end{align}
For any $r > 0$, we define the operator
\[
\mathfrak M_rf(x)=\Phi_r*f(x),
\]
where $\Phi_r(y)=r^{-{\rm tr}(A)}\Phi(r^{-A}y)$ and $A$ is the
diagonal $d \times d$ matrix as defined in \eqref{eq:25}. Then by the classical vector-valued maximal 
estimates (see \cite{bigs}) we have
\begin{align}
  \label{eq:47}
    \Big\|\Big(\sum_{t\in\NN}\sup_{n\in\ZZ}\big|\mathfrak
    M_{2^n}f_t\big|^2\Big)^{1/2}\Big\|_{L^p}\le C_p
    \Big\|\big(\sum_{t\in\NN}|f_t\big|^2\big)^{1/2}\Big\|_{L^p}.
\end{align}
Hence, for the proof of Theorem \ref{thm:17}, it is sufficient to show that
\begin{align}
  \label{eq:48}
    \Big\|\Big(\sum_{t\in\NN}\sum_{n\in\ZZ}\big|\big(\mathcal
    M_{2^n}-\mathfrak M_{2^n}\big)f_t\big|^2\Big)^{1/2}\Big\|_{L^p}\le C_p
    \Big\|\big(\sum_{t\in\NN}|f_t\big|^2\big)^{1/2}\Big\|_{L^p}.
\end{align}
Appealing to Lemma \ref{lem:30}, the proof of \eqref{eq:48} will be completed if we can prove
that for every $p\in(1, \infty)$ and $f\in L^p\big(\RR^d\big)$ we have
\begin{align}
  \label{eq:49}
	\Big\|\sum_{n\in\ZZ}\varepsilon_n(\omega)\big(\mathcal
    M_{2^n}-\mathfrak M_{2^n}\big)f\Big\|_{L^p}\le C_p\|f\|_{L^p}
\end{align}
for every Rademacher sequence $(\varepsilon_n(\omega): n\in\NN_0)$, where $\omega\in [0,
1]$ and $\varepsilon_n(\omega)\in\{-1, 1\}$.

While proving \eqref{eq:49}, without loss of generality, we may assume that $p \geq 2$.
Let $U_n=\varepsilon_n(\omega)\big(\mathcal M_{2^n}-\mathfrak M_{2^n}\big)$. By $S_j$ we
denote a Littlewood--Paley projection $\calF(S_j g)(\xi)=\phi_j(\xi)\calF g (\xi)$ associated
with a smooth partition of unity $(\phi_j: j\in\ZZ)$ on $\RR^d\setminus\{0\}$, such that
for each $j \in \ZZ$ we have $0\le \phi_j\le 1$, and
	\[
		\supp{\phi_j}
		\subseteq
		\big\{\xi\in\RR^d: 2^{-1} < \abs{2^{jA}\xi}_{\infty} < 2 \big\},
	\]
and 
\[
\sum_{j\in\ZZ}\phi_j(\xi)^2=1, \quad \text{for} \quad \xi\in \RR^d\setminus\{0\}.
\]
Note that 
\begin{align*}
	\Big\|\sum_{n\in\ZZ}\varepsilon_n(\omega)\big(\mathcal
    M_{2^n}-\mathfrak M_{2^n}\big)f\Big\|_{L^p}
	&\le\sum_{j\in\ZZ}\Big\|\sum_{n\in\ZZ}S_{j+n}U_nS_{j+n}f\Big\|_{L^p} \\
	&\lesssim\sum_{j\in\ZZ}
	\Big\|\Big(\sum_{n\in\ZZ}\big|U_nS_{j+n}f\big|^2\Big)^{1/2}\Big\|_{L^p}.
\end{align*}
Next, we show that there is a constant $\delta_p>0$ such that
\begin{align}
  \label{eq:51}
  \Big\|\Big(\sum_{n\in\ZZ}\big|U_nS_{j+n}f\big|^2\Big)^{1/2}\Big\|_{L^p}
	\lesssim 2^{-\delta_p|j|}\|f\|_{L^p}.
\end{align}
Set $\mathcal M_*f=\sup_{r>0}|\mathcal M^*_rf|$ and $\mathfrak M_*f=\sup_{r>0}|\mathfrak M^*_rf|$, where 
$\mathcal M^*_r$ and $\mathfrak M^*_r$ is the adjoint operator to $\mathcal M_r$ and
$\mathfrak M_r$, respectively. Let $g\in L^q\big(\RR^d\big)$ be such that $\|g\|_{L^q}\le 1$
and $g\ge0$, where $q=(p/2)'$. Then by the Cauchy--Schwarz inequality
\begin{align*}
  \int_{\RR^d}\sum_{n\in\ZZ}\big|U_nS_{j+n}f(x)\big|^2g(x){\: \rm  d}x
	& \le\int_{\RR^d}\sum_{n\in\ZZ}|S_{j+n}f(x)|^2\big(\mathcal
  M_*+\mathfrak M_*\big)g(x){\: \rm  d}x\\
& \le  \bigg\|\Big(\sum_{n\in\ZZ}\big|S_{j+n}f\big|^2\Big)^{1/2}\bigg\|_{L^p}^2
\big\|\big(\mathcal
  M_*+\mathfrak M_*\big)g\big\|_{L^q}
\lesssim \|f\|_{L^p}^2\|g\|_{L^q}
\end{align*}
since for all $q\in(1, \infty]$ there is a constant $C_q>0$ such that for all $g\in L^p\big(\RR^d\big)$,
\[
	\big\|\mathcal
  	M_*g\big\|_{L^q}+\big\|\mathfrak M_*g\big\|_{L^q}\le C_q\|g\|_{L^q}.
\]
Thus we have proven that
\begin{align}
\label{eq:52}
  \bigg\|\Big(\sum_{n\in\ZZ}\big|U_nS_{j+n}f\big|^2\Big)^{1/2}\bigg\|_{L^p}\lesssim \|f\|_{L^p}.
\end{align}
Now we refine the estimate \eqref{eq:52} by showing that there is $\delta>0$ such that
\begin{align}
\label{eq:53}
  \bigg\|\Big(\sum_{n\in\ZZ}\big|U_nS_{j+n}f\big|^2\Big)^{1/2}\bigg\|_{L^2}\lesssim 2^{-\delta|j|}\|f\|_{L^2}.
\end{align}
Then interpolation of \eqref{eq:52} with \eqref{eq:53} will imply
\eqref{eq:51} and the proof of Theorem \ref{thm:17} will be
completed. Let
\[
	m_{2^n}(\xi)=\frac{1}{|B_1|}\int_{B_1}e^{2\pi i \xi\cdot\calQ(2^ny)}{\: \rm d}y 
	\qquad\text{and}\qquad
	\mathfrak m_{2^n}(\xi)=\calF\Phi(2^{nA}\xi),
\]
be the multipliers associated with the averages $\mathcal M_{2^n}$
and $\mathfrak M_{2^n}$, respectively. Then
\[
  |m_{2^n}(\xi)-\mathfrak m_{2^n}(\xi)|\lesssim \min\big\{|2^{nA}\xi|_{\infty},
  |2^{nA}\xi|_{\infty}^{-1/d}\big\}.
\]
Hence, by Plancherel's theorem and the assumption on the supports for
$\phi_{n+j}$ we obtain
\begin{align*}
  \Big\|\Big(\sum_{n\in\ZZ}\big|U_nS_{j+n}f\big|^2\Big)^{1/2}\Big\|_{L^2}^2
	&=\int_{\RR^d}\sum_{n\in\ZZ}
	|m_{2^n}(\xi)-\mathfrak m_{2^n}(\xi)|^2|\phi_{j+n}(\xi)|^2|\mathcal F f(\xi)|^2{\: \rm d}\xi\\
	& 
	\lesssim 2^{-2\delta |j|}\int_{\RR^d}\sum_{n\in\ZZ}|\phi_{j+n}(\xi)|^2|\mathcal F f(\xi)|^2{\: \rm d}\xi
	\lesssim 2^{-2\delta |j|}\|f\|_{L^2}^2,
\end{align*}
for some $\delta>0$ as desired.

\subsection{Proof of Theorem \ref{thm:18}}
For each kernel $K$ satisfying \eqref{eq:29} and \eqref{eq:30} we have the decomposition \eqref{eq:8}.
Therefore for any $f \geq 0$, we have the pointwise bound
\[
  \sup_{r>0}|\mathcal T_rf(x)|\lesssim   \sup_{r>0}\mathcal M_rf(x)+
  \sup_{n\in\ZZ}\Big|\sum_{j\ge n}T_jf(x)\Big|,
\]
where
\[
	T_jf(x)=\int_{\RR^k}f\big(x-\calQ(y)\big)K_j(y){\: \rm d}y=\mu_{2^j}*f(x).
\]
Hence, Theorem \ref{thm:17} reduce the matters to proving that for every $p\in(1, \infty)$ 
there exists a constant $C_p>0$ such that for every sequence $(f_t: t\in\NN)\in
  L^p\big(\ell^2\big(\RR^{d_0}\big)\big)$ we have
	\[
    \Big\|\Big(\sum_{t\in\NN}\sup_{n\in\ZZ}\Big|\sum_{j\ge n}T_jf_t\Big|^2\Big)^{1/2}\Big\|_{L^p}\le C_p
    \Big\|\big(\sum_{t\in\NN}|f_t\big|^2\big)^{1/2}\Big\|_{L^p}.
	\]
Let
\[
	Tf(x)=\sum_{j\in\ZZ} T_jf(x).
\]
Recall that (see e.g. \cite{bigs})
\[
	\|Tf\|_{L^p}\lesssim \|f\|_{L^p}.
\]
As in \cite{DuoRdF} we decompose
\begin{align*}
\sum_{j\ge n}T_jf&=\Phi_{2^n}*\big(Tf-\sum_{j<n}T_jf\big)+(\delta_0-\Phi_{2^n})*\sum_{j\ge n}T_jf\\
	&=\Phi_{2^n}*Tf-\big(\Phi_{2^n}*\sum_{j<n}\mu_{2^j}\big)*f+\sum_{j\ge 0}(\delta_0-\Phi_{2^n})*\mu_{2^{j+n}}*f.
\end{align*}
Observe that by \eqref{eq:47} and \eqref{eq:69} we have
\begin{align*}
    \Big\|\Big(\sum_{t\in\NN}\sup_{n\in\ZZ}\big|\Phi_{2^n}*Tf_t\big|^2\Big)^{1/2}\Big\|_{L^p}
	&\lesssim 
	\Big\|\big(\sum_{t\in\NN}\big|Tf_t\big|^2\big)^{1/2} \Big\|_{L^p}\\
	&\lesssim
    \Big\|\big(\sum_{t\in\NN}|f_t\big|^2\big)^{1/2}\Big\|_{L^p}.
\end{align*}
Since the function $\Phi_{2^n}*\sum_{j<n}\mu_{2^j}$ defines a Schwartz 
function, for every $N\in\NN_0$ we have
\[
	\Big|\Phi_{2^n}*\sum_{j<n}\mu_{2^j}(x)\Big|
	\lesssim_N 2^{-{\rm tr}(A)n}\big(1+|2^{-nA}x|^2\big)^{-N},
\]
thus by the classical vector-valued estimates (see e.g. \cite{bigs}) we get
\[
  \Big\|\Big(\sum_{t\in\NN}\sup_{n\in\ZZ}\big|\big(\Phi_{2^n}*\sum_{j<n}\mu_{2^j}\big)*f_t\big|^2\Big)^{1/2}\Big\|_{L^p}
	\lesssim
  \Big\|\big(\sum_{t\in\NN}|f_t\big|^2\big)^{1/2}\Big\|_{L^p}.
\]
Now, it remains to prove that there exists $\delta_p>0$ such that
 \begin{align}
   \label{eq:197}
	\Big\|\Big(\sum_{t\in\NN}\sup_{n\in\ZZ}\big|(\delta_0-\Phi_{2^n})*\mu_{2^{j+n}}*f_t\big|^2\Big)^{1/2}\Big\|_{L^p}
	\lesssim2^{-\delta_p j}
	\Big\|\big(\sum_{t\in\NN}|f_t\big|^2\big)^{1/2}\Big\|_{L^p}.
 \end{align}
For $p\in(1, \infty)$, due to Theorem \ref{thm:17} and \eqref{eq:47} we have
\begin{align}
\label{eq:198}
	\Big\|\Big(\sum_{t\in\NN}\sup_{n\in\ZZ}\big|(\delta_0-\Phi_{2^n})*\mu_{2^{j+n}}*f_t\big|^2\Big)^{1/2}\Big\|_{L^p}
	\lesssim
  	\Big\|\big(\sum_{t\in\NN}|f_t\big|^2\big)^{1/2}\Big\|_{L^p}.
 \end{align}
For $p=2$ we show that there exists $\delta>0$ such that
\begin{align}
   \label{eq:199}
	\Big\|\Big(\sum_{t\in\NN}\sup_{n\in\ZZ}\big|(\delta_0-\Phi_{2^n})*\mu_{2^{j+n}}*f_t\big|^2\Big)^{1/2}\Big\|_{L^2}
	\lesssim2^{-\delta j}
  	\Big\|\big(\sum_{t\in\NN}|f_t\big|^2\big)^{1/2}\Big\|_{L^2}.
 \end{align}
Then interpolation \eqref{eq:198} with \eqref{eq:199} will establish \eqref{eq:197} and the proof of Theorem \ref{thm:18}
will be completed. For the proof of \eqref{eq:199}, it suffices to show
\[ 
  \big\|\sup_{n\in\ZZ}\big|(\delta_0-\Phi_{2^n})*\mu_{2^{j+n}}*f\big|\big\|_{L^2}\lesssim
  2^{-\delta j}\|f\|_{L^2}.
\]
Since
\[
	\big|1-\mathcal F\Phi\big(2^{nA}\xi\big)\big|\lesssim \big|2^{nA}\xi\big|_{\infty}^{1/d},
\]
and
\[
	\big|\calF\mu_{2^{j+n}}(\xi)\big|\lesssim \min\big\{\big|2^{(j+n)A}\xi\big|_{\infty},
  	\big|2^{(j+n)A}\xi\big|_{\infty}^{-1/d}\big\},
\]
by Plancherel's theorem we see that
\begin{align*}
	& \big\|\sup_{n\in\ZZ}\big|(\delta_0-\Phi_{2^n})*\mu_{2^{j+n}}*f\big|\big\|_{L^2}^2
	\le
	\int_{\RR^d}\sum_{n\in\ZZ}\big|\big(1-\mathcal F\Phi\big(2^{nA}\xi\big)\big)\calF\mu_{2^{j+n}}(\xi)\big|^2
	|\calF f(\xi)|^2{\:\rm d}\xi\\
	&\qquad \qquad \lesssim
	\int_{\RR^d}\sum_{n\in\ZZ}\big|2^{nA}\xi\big|_{\infty}^{1/d}
	\big|2^{(j+n)A}\xi\big|_{\infty}^{-1/d}\min\big\{\big|2^{(j+n)A}\xi\big|_{\infty},
	 \big|2^{(j+n)A}\xi\big|_{\infty}^{-1/d}\big\}|\calF f(\xi)|^2{\:\rm d}\xi\\
	& \qquad \qquad \lesssim  2^{-j/d} \int_{\RR^d}\sum_{n\in\ZZ}\min\big\{\big|2^{(j+n)A}\xi\big|_{\infty},
  	\big|2^{(j+n)A}\xi\big|_{\infty}^{-1/d}\big\}|\calF f(\xi)|^2{\:\rm d}\xi,
\end{align*}
which is bounded by $2^{-j/d}\|f\|_{L^2}^2$ as claimed. This completes the proof of Theorem \ref{thm:18}.


\end{document}